\documentclass[12pt]{amsart}

\usepackage[text={400pt,660pt},centering]{geometry}

\usepackage{color}
\usepackage{esint,amssymb,mathrsfs}
\usepackage{graphicx}
\usepackage{MnSymbol}
\usepackage{mathtools}
\usepackage[colorlinks=true, pdfstartview=FitV, linkcolor=blue, citecolor=blue, urlcolor=blue,pagebackref=false]{hyperref}
\usepackage{microtype}

\usepackage{bm,dsfont}
\usepackage{etoolbox}
\AtBeginEnvironment{quote}{\small}
\usepackage{scalerel} 
\usepackage[font={footnotesize}]{caption}

\definecolor{darkgreen}{rgb}{0,0.5,0}
\definecolor{darkblue}{rgb}{0,0,0.7}
\definecolor{darkred}{rgb}{0.9,0.1,0.1}
\definecolor{yellow}{rgb}{0.9,0.9,0.1}

\newtheorem{proposition}{Proposition}
\newtheorem{theorem}[proposition]{Theorem}
\newtheorem{lemma}[proposition]{Lemma}
\newtheorem{corollary}[proposition]{Corollary}

\theoremstyle{remark}
\newtheorem{remark}[proposition]{Remark}

\theoremstyle{definition}

\numberwithin{equation}{section}
\numberwithin{proposition}{section}

\newcommand{\Z}{\mathbb{Z}}
\newcommand{\N}{\mathbb{N}}
\newcommand{\R}{\mathbb{R}}

\renewcommand{\P}{\mathbb{P}}

\newcommand{\Zd}{\mathbb{Z}^d}
\newcommand{\Rd}{{\mathbb{R}^d}}

\newcommand{\ep}{\varepsilon}

\renewcommand{\subset}{\subseteq}
\renewcommand{\fint}{\strokedint}

\renewcommand{\tilde}{\widetilde}
\renewcommand{\hat}{\widehat}

\newcommand{\indc}{\mathds{1}}

\renewcommand{\a}{\mathbf{a}}

\newcommand{\ahom}{\overline{\a}}

\newcommand{\A}{\mathbb{A}}

\renewcommand{\ahom}{\overline{\mathbf{a}}}

\newcommand{\Id}{\mathbf{I}}

\begin{document}

\title[Analyticity for periodic elliptic equations]{Large-scale analyticity and unique continuation for periodic elliptic equations}

\begin{abstract}
We prove that a solution of an elliptic operator with periodic coefficients behaves on large scales like an analytic function, in the sense of approximation by polynomials with periodic corrections. Equivalently, the constants in the large-scale~$C^{k,1}$ estimate scale exponentially in~$k$, just as for the classical estimate for harmonic functions. 
As a consequence, we characterize entire solutions of periodic, uniformly elliptic equations which exhibit growth like~$O(\exp(\delta|x|))$ for small~$\delta>0$. 
The large-scale analyticity also implies quantitative unique continuation results, namely a three-ball theorem with an optimal error term as well as a proof of the nonexistence of~$L^2$ eigenfunctions at the bottom of the spectrum. 
\end{abstract}

\author[S. Armstrong]{Scott Armstrong}
\address[S. Armstrong]{Courant Institute of Mathematical Sciences, New York University, USA}
\email{scotta@cims.nyu.edu}

\author[T. Kuusi]{Tuomo Kuusi}
\address[T. Kuusi]{Department of Mathematics and Statistics, P.O. Box 68 (Gustaf H\"allstr\"omin katu 2), FI-00014 University of Helsinki, Finland}
 \email{tuomo.kuusi@helsinki.fi}
 
\author[C. Smart]{Charles Smart}
\address[C. Smart]{Department of Mathematics, The University of Chicago, 5734 S. University Avenue, Chicago, Illinois 60637.}
\email{smart@math.uchicago.edu}

\date{\today}

\maketitle

\setcounter{tocdepth}{1}

\section{Introduction}

\subsection*{Motivation and informal summary of results}
We consider the linear, divergence-form, uniformly elliptic equation
\begin{equation}
\label{e.pde}
-\nabla\cdot \a \nabla u = 0 \quad \mbox{in} \ U \subseteq\Rd,
\end{equation}
where $\a(\cdot)$ is a measurable, $\Zd$--periodic map from $\Rd$ into the set of~$d$-by-$d$ matrices satisfying a uniform ellipticity assumption. Precisely, we assume there exists a constant $\Lambda\in [1,\infty)$ such that $\a \in L^\infty\left(\Rd; \R^{d\times d} \right)$ satisfies
\begin{equation}
\label{e.a.ue}
\left| \xi \right|^2 \leq \xi \cdot \a(x)\xi
\quad \mbox{and} \quad
\left| \eta \cdot \a(x)\xi \right| 
\leq \Lambda \left| \eta \right| \left| \xi \right|,
\quad
\forall x,\xi,\eta \in\Rd, 
\end{equation}
and
\begin{equation}
\label{e.periodicity}
\a(\cdot + z) = \a, \quad \forall z\in\Zd. 
\end{equation}
In particular, we make no assumption of regularity on the coefficient field~$\a(\cdot)$ beyond measurability. 

\smallskip

This paper concerns the behavior of solutions of~\eqref{e.pde} on length scales much larger than the unit scale on which the coefficients oscillate. The particular focus is on regularity and quantitative unique continuation properties of solutions. Our main result (see Theorem~\ref{t.analyticity} below) states that, on large scales, solutions of~\eqref{e.pde} possess higher-order regularity properties analogous to the classical pointwise estimates for the derivatives of harmonic functions which imply analyticity. This roughly means that a solution of~\eqref{e.pde} can be locally approximated by a certain family of ``heterogeneous polynomials'' (polynomials with periodic corrections) with the same precision as the approximation of a harmonic function by its Taylor polynomial. We call this a ``large-scale analyticity'' estimate. 

\smallskip

Large-scale regularity estimates, which play a fundamental role in the theory of homogenization, originated in the work of Avellaneda and Lin~\cite{AL1,AL2}. A qualitative version of large-scale $C^{k,1}$ estimates were proved in~\cite{AL2} for every~$k\in\N$ in the form of a Liouville theorem characterizing all entire solutions with at most polynomial growth. The large-scale analyticity estimate we prove can be considered to be a quantification of the dependence in~$k$ of the prefactor constant in the~$C^{k,1}$ estimate (it is exponential in~$k$) and of the ``minimal scale'' on which it is valid (it is linear in $k$). Moreover, each  of these estimates are optimal for~$k\gg1$.

\smallskip

We also present three consequences of the large-scale analyticity estimate. The first is a Liouville-type result (see Theorem~\ref{t.Liouville} below) which characterizes, for a small exponent~$\delta>0$, the set of solutions of~\eqref{e.pde} in the whole space $\Rd$ which grow at most like~$\exp(\delta|x|)$ as~$|x|\to \infty$. We show that any such solution can be written as an infinite series of heterogeneous polynomials, analogous to a power series representation. This can be seen as a qualitative version of large-scale analyticity. 

\smallskip

The second application is a quantitative unique continuation result in the form of a so-called ``three-ball theorem'' (see Theorem~\ref{t.quc}). It states roughly that given, three balls $B_r\subseteq B_s\subseteq B_R$ with $R/s = s/r$ and a solution $u$ of~\eqref{e.pde} in~$B_R$, the ratio $\| u \|_{L^2(B_R)} / \| u \|_{L^2(B_s)}$ is controlled by $\| u \|_{L^2(B_s)} / \| u \|_{L^2(B_r)}$, plus an error term which depends on the scale~$r$. Recall that the strong unique continuation property is well known to hold for elliptic equations with Lipschitz coefficients~\cite{AKS,GL}, but is false, in general, even for equations with coefficients belonging to~$C^{0,\alpha}$ for every $\alpha\in (0,1)$: see~\cite{P,Mi,Fil}. Therefore quantitative unique continuation estimates necessarily depend on the Lipschitz norm of the coefficients and degenerate on large length scales. Here we prove the opposite: a quantitative unique continuation result which becomes stronger as the length scale becomes larger and is completely useless below the unit scale.

\smallskip

One of the main motivations for studying unique continuation properties lies in their implications for the spectrum of the operator. It is a long-standing conjecture that the spectrum of the elliptic operator~$-\nabla \cdot \a \nabla $ with periodic \emph{and sufficiently smooth} coefficients~$\a(\cdot)$ must be absolutely continuous. For a complete discussion of the history and recent progress on this problem, which is still wide-open, we refer to the survey of Kuchment~\cite[Section 6]{Ku}. We mention only that the only progress made on this question for equations with variable coefficients is the work of Friedlander~\cite{F}, who proved the conjecture under some symmetry assumptions on the coefficients which are unfortunately rather restrictive. In the negative direction, Filonov~\cite{Fil} constructed an explicit example of a coefficient field~$\a(\cdot)$ belonging to~$C^{0,\alpha}$ for every $\alpha<1$ such that the operator $-\nabla \cdot \a\nabla$ admits a nontrivial and compactly-supported eigenfunction (note that compact support allows the example to be periodized). 

\smallskip

The third consequence of the large-scale analyticity estimate we present asserts that the operator~$-\nabla \cdot\a\nabla$ has no $L^2$ eigenfunctions near the bottom of the spectrum (see Theorem~\ref{t.bottom} below). We emphasize that this result is valid without regularity assumptions on~$\a(\cdot)$ and therefore, in view of the counterexample of~\cite{Fil} mentioned above, cannot be improved. Indeed, the result of~\cite{Fil} shows the optimality of each of the four theorems presented below.

\subsection*{Statement of the main results}
Before presenting the main results, we must introduce some notation. For further notation, including the notation for multiindices and tensors, see Section~\ref{ss.prelim}.
As we find it convenient to work with cubes rather than balls, we denote, for every $r>0$, 
\begin{equation}
Q_r:= \left( -\frac12r,\frac12 r \right)^d. 
\end{equation}
The tensor of~$k$th order partial derivatives of a function~$u$ is denoted~$\nabla^k u$. For $u \in L^p(U)$ with $p\in[1,\infty)$ and $0<|U|<\infty$, we denote the volume-normalized $L^p(U)$ norm by
\begin{equation} 
\label{e.L2norm}
\left\| u \right\|_{\underline{L}^p(U)} := \left( \fint_{U} |u(x)|^p \, dx \right)^{\frac1p} := \left( \frac{1}{|U|}\int_{U} |u(x)|^p \, dx \right)^{\frac1p} .
\end{equation}
For each $m\in\N$, we denote the linear space of real-valued, homogeneous polynomials of order~$m$ on $\Rd$ by
\begin{equation}
\mathbb{P}_m^*:=
\left\{ 
w \, : \,
w =\sum_{|\alpha|=m} c_\alpha x^\alpha, \ c_\alpha\in\R
\right\}
\end{equation}
and the linear of all real-valued polynomials of order at most $m$ by 
\begin{equation}
\mathbb{P}_m:=
\left\{ 
w \, : \,
w = \sum_{n=0}^m w_n, \ w_n \in \mathbb{P}_n^*
\right\}.
\end{equation}
For convenience we also denote $\P_{-1} := \P_{-2} := \{ 0 \}$. 

\smallskip

In order to present the statement of the large-scale analyticity result, we briefly describe the finite-dimensional vector space~$\A_m$ of ``heterogeneous polynomials of degree at most $m\in\N$'' and refer the reader to Section~\ref{s.correctors} for more details. The result of~\cite{AL2} states that, for every $m\in\N$, $p\in \P_{m-2}$ and $u\in H^1_{\mathrm{loc}}(\Rd)$ satisfying the equation
\begin{equation}
-\nabla \cdot \a\nabla u = p \quad \mbox{in} \ \Rd
\end{equation}
and the growth condition
\begin{equation}
\limsup_{r\to \infty} 
r^{-(m+1)}
\left\| u \right\|_{\underline{L}^2(Q_r)} 
= 0,
\end{equation}
one can find a polynomial $q \in \P_{m}$ such that 
\begin{equation}
\label{e.uformq}
u(x) = \sum_{k=0}^m \nabla^kq(x) : \phi^{(k)}(x),
\end{equation}
where $\phi^{(k)}$ is the \emph{tensor of $k$th order correctors}, which are constructed in Section~\ref{s.correctors}, and the colon denotes the tensor contraction, see~\eqref{e.colon}. We think of such solutions as the correct notion of ``intrinsic polynomials of degree at most~$m$'' and so we define, for each~$m\in\N$,
\begin{equation}
\mathbb{A}_m:=
\left\{ 
\psi = \sum_{k=0}^m 
\left( \nabla^k q : \phi^{(k)} \right)
\,:\, 
q\in \P_m
\right\}. 
\end{equation}
Given~$u$ of the form~\eqref{e.uformq}, we have that $-\nabla \cdot \a\nabla u = \mathscr{A}q$, where $\mathscr{A}$ is the higher-order homogenized operator. It has the form
\begin{equation*}
\mathscr{A} u := 
\sum_{n=2}^\infty 
\ahom^{(n)} : \nabla^n u,
\end{equation*}
where the tensors $\ahom^{(n)}$ are the higher-order homogenized tensors which arise in the construction of the higher-order correctors (see Section~\ref{s.correctors}). We note that $\ahom^{(2)}=:\ahom$ is the familiar homogenized matrix arising in classical homogenization. It is the equation $-\mathscr{A} u = 0$ which should be considered as the ``true'' homogenized equation, rather than $-\nabla\cdot\a\nabla u =0$, which is merely the leading order approximation of the former. We also denote the subset of~$\A_m$ consisting of~$\a(\cdot)$--harmonic functions by 
\begin{equation}
\label{e.Am0.intro}
\mathbb{A}^0_m:=
\left\{ 
u = \sum_{k=0}^m 
\left( \nabla^k q : \phi^{(k)} \right)
\,:\, 
q\in \P_m, \ \mathscr{A}q = 0
\right\} .  
\end{equation}

\smallskip

The classical pointwise estimate for harmonic functions states that there exists~$C(d)<\infty$ such that, for every harmonic function~$u$ in~$Q_R$ and~$m\in\N$,  
\begin{equation}
\label{e.harm1}
\left| \nabla^m u (0) \right| 
\leq 
\left( \frac{Cm}{R} \right)^m \left\| u \right\|_{\underline{L}^1(Q_R)}.
\end{equation}
We can restate this in terms of polynomial approximation as follows: there exists~$C(d)<\infty$ such that, for every harmonic function~$u$ in~$B_R$ and~$m\in\N$, there exists $p\in\P_m$ such that, for every $r\in (0,R)$, 
\begin{equation}
\label{e.harm2}
\left\| u - p \right\|_{L^\infty(Q_r)} 
\leq
\left( \frac{Cr}{R} \right)^m \left\| u \right\|_{\underline{L}^1(Q_R)}.
\end{equation}
Indeed, one may take~$p$ to be the $m$th order Taylor polynomial of~$u$ at the origin (which we note is also harmonic) and combine~\eqref{e.harm1} with Taylor's theorem. 

\smallskip

The main result of the paper is the following generalization of~\eqref{e.harm2} to equations with variable, periodic coefficients.

\begin{theorem}[Large-scale analyticity]
\label{t.analyticity}
There exists~$C(d,\Lambda)<\infty$ such that, for every $m\in\N$ with $m\geq 0$, $R\in [2Cm,\infty)$ and solution $u\in H^1(Q_R)$ of
\begin{equation}
-\nabla \cdot \a\nabla u = 0 \quad \mbox{in} \ Q_R,
\end{equation}
there exists $\psi \in \A_{m}^0$ such that, for every $r\in \left[Cm, R\right]$,
\begin{equation}
\label{e.Cm1}
\left\| u - \psi \right\|_{\underline{L}^2(Q_r)} 
\leq 
\left( \frac{Cr}{R} \right)^{m+1} \left\| u \right\|_{\underline{L}^2(Q_R)}. 
\end{equation}
\end{theorem}

The main novelty in the statement of Theorem~\ref{t.analyticity} is the fact that the constant~$C$ on the right side of~\eqref{e.Cm1} depends only on $(d,\Lambda)$ and, in particular, does not depend on~$m$. It is well-known result, essentially due to~\cite{AL2} (see for instance~\cite[Theorem 3.6]{AKMbook} for the complete statement), that~\eqref{e.Cm1} is valid if we replace the right side by 
\begin{equation}
C(m,d,\Lambda) \left( \frac{r}{R} \right)^{m+1} 
\left\| u \right\|_{\underline{L}^2(Q_R)}. 
\end{equation}
Since this estimate is qualitative in its dependence in~$m$, it unfortunately gives no information regarding the ratio $r/R$ of length scales on which it is useful. By carefully quantifying the dependence on~$m$ from previous arguments, one at best obtains $C(m,d,\Lambda) =  \exp\left( C(d,\Lambda){m^2} \right)$, which is much worse than what is given by Theorem~\ref{t.analyticity}, namely $C(m,d,\Lambda) = \exp(C(d,\Lambda)m)$. Since the latter is the same as observed in the classical estimate~\eqref{e.harm2} for harmonic functions, it is therefore optimal. 

\smallskip

The reason we use the modifier ``large-scale'' in describing Theorem~\ref{t.analyticity} is due to the restriction~$r \geq Cm$ in~\eqref{e.Cm1}, which does not appear in~\eqref{e.harm2}. This restriction is necessary and sharp in the sense that, without at least an assumption of Lipschitz regularity for~$\a(x)$, the statement of the theorem would be false if the restriction~$r\geq Cm$ is replaced by~$r \geq \ep m$ for sufficiently small~$\ep$. This can be seen\footnote{One can add a dummy variable $x_{d+1}$ and multiply by $\exp(\lambda^{\frac12}x_{d+1})$ to make an eigenfunction with eigenvalue $\lambda$ into a solution with zero right-hand side in one more dimension. } from the example of the compactly supported eigenfunction constructed in~\cite{Fil}. 

\smallskip

Perhaps a better way to make the motivation behind the terminology clear is to present the following rephrasing of Theorem~\ref{t.analyticity}.

\begin{corollary}
\label{c.analyticity.exp}
There exists~$c(d,\Lambda)\in (0,1]$ such that, for every~$R\in \left[c^{-1},\infty\right)$ and solution $u\in H^1(Q_R)$ of
\begin{equation}
-\nabla \cdot \a\nabla u = 0 \quad \mbox{in} \ Q_R,
\end{equation}
there exists $\psi \in \A_{\lfloor R \rfloor }^0$ such that, for every $r\in \left[1, cR\right]$,
\begin{equation}
\label{e.Cm1exp}
\left\| u - \psi \right\|_{\underline{L}^2(Q_r)} 
\leq 
\exp(-cR)
\left\| u \right\|_{\underline{L}^2(Q_R)}. 
\end{equation}
\end{corollary}

We see from~\eqref{e.Cm1exp} that our large-scale analyticity does not yield ``literal'' analyticity in the sense that an arbitrary solution can be expressed locally as a ``power series,'' interpreted as a limit of elements of $\A_m$ with $m\to \infty$. What Corollary~\ref{c.analyticity.exp} says however is that this is valid up to an error which is exponentially small in the length scale. The example of~\cite{Fil} implies moreover that this is optimal.

\smallskip

As a first application of Theorem~\ref{t.analyticity}, we obtain a characterization of entire solutions of~\eqref{e.pde} which have growth at most like a ``slow'' exponential. That is, we will show that a solution in the whole space which does not grow faster than $\exp(c|x|)$ can indeed be written as a ``power series''---a result which is not very surprising in view of Corollary~\ref{c.analyticity.exp}. For each~$\delta>0$, let us denote 
\begin{equation*}
\mathbb{P}_\infty(\delta) 
:= 
\left\{ 
w \, : \,
w = \sum_{n=0}^\infty w_n, \ w_n \in \mathbb{P}_n^*, \
\sum_{n=0}^\infty 
\delta^{-n} \left| \nabla^n w_n(0) \right| 
< \infty
\right\}.
\end{equation*}
Using bounds on the $k$th order correctors and homogenized tensors which (are proved in Section~\ref{s.correctors} and) state that these grow at most like $C^k$, it follows that, for every~$\delta>0$ sufficiently small and $q \in \P_\infty(\delta)$, the series
\begin{equation}
u := 
\sum_{k=0}^\infty \nabla^kq:\phi^{(k)}
\end{equation}
is absolutely convergent (locally uniformly in~$\Rd$) and the expression~$\mathscr{A}q$ is well-defined. For such~$\delta$,  we may define the space 
\begin{equation}
\A_\infty(\delta):=
\left\{ 
\psi = \sum_{k=0}^\infty 
\left( \nabla^k q:\phi^{(k)} \right)
\, :\,
q \in \P_\infty(\delta) 
\right\}
\end{equation}
as well as $\A_\infty^0(\delta)$ in analogy with~\eqref{e.Am0.intro}. The following theorem asserts that this space contains all solutions of~\eqref{e.pde} in~$\Rd$ which grow like a slow exponential.

\begin{theorem}[Liouville theorem]
\label{t.Liouville}
There exists~$C(d,\Lambda)<\infty$ and~$\delta_0(d,\Lambda)>0$ such that, for every $u\in H^1_{\mathrm{loc}}(\Rd)$ and $\delta \in (0,\delta_0]$ satisfying
\begin{equation}
-\nabla \cdot \a\nabla u = 0 \quad \mbox{in}\ \Rd 
\end{equation}
and
\begin{equation}
\limsup_{r\to \infty}\ \exp\left( - \delta r \right) \left\| u \right\|_{\underline{L}^2(Q_r)} 
< \infty,
\end{equation}
we have that $u \in \A_\infty(C\delta)$. 
\end{theorem}

\smallskip

We next present a quantitative unique continuation result, which comes in the form of a three-ball theorem. 
To our knowledge, the first quantitative unique continuation result for periodic coefficients appeared in the recent work of Lin and Shen~\cite{LS}, who proved a large-scale doubling condition for solutions of~\eqref{e.pde}. Subsequently, a statement very similar to our Theorem~\ref{t.quc} was obtained in a  paper of Kenig and Zhu~\cite{KZ}. Their result does not have the restriction that the ratio of the radii be small, but their error term has a factor of~$R^{-\beta}$, for an exponent~$\beta \in (0,1)$, rather than the sharp exponential factor $\exp(-cr)$ in~\eqref{e.quc}.

\begin{theorem}[Three-ball theorem]
\label{t.quc}
For each $\alpha \in \left( 0,\tfrac12 \right)$, 
there exist~$c(d,\Lambda)>0$ and $\theta(\alpha, d,\Lambda) \in \left(0,\tfrac12\right]$ such that, for every $R,r,s\in [2,\infty)$ with 
\begin{equation}
\frac{r}{s} = \frac{s}{R} \in (0,\theta]
\end{equation}
and $u \in H^1(B_R)$ satisfying 
\begin{equation}
-\nabla \cdot \a\nabla u = 0\quad \mbox{in} \ B_R, 
\end{equation}
we have the estimate
\begin{equation}
\label{e.quc}
\left\| u  \right\|_{\underline{L}^2(B_s)}
\leq
\left\| u \right\|_{\underline{L}^2(B_r)}^{\alpha}
\left\| u \right\|_{\underline{L}^2(B_R)}^{1-\alpha}
+
\exp(-cr) \left\| u \right\|_{\underline{L}^2(B_R)}.
\end{equation}
\end{theorem}

The error term in the estimate~\eqref{e.quc} is sharp in the sense that, without at least an assumption of Lipschitz regularity on the coefficients~$\a(x)$, the multiplicative factor must at least~$\exp( -Cr )$. This is due to same example of~\cite{Fil} mentioned several times above. It remains a very interesting open question whether, under suitable smoothness assumptions on the coefficient~$\a(\cdot)$, the~$c$ inside the exponential in~\eqref{e.quc} may be replaced by any large constant~$A>1$. (If this could be affirmatively resolved, then the restriction~$\lambda \leq \lambda_0$ in Theorem~\ref{t.bottom} below could be also be removed.) 
We remark that, as can be observed from the proof of Theorem~\ref{t.quc} in Section~\ref{s.conseq}, the size of~$c$ is closely related to the constant~$C$ in the restriction $r\geq Cm$ in Theorem~\ref{t.analyticity}. The optimal exponent~$\alpha$ in~\eqref{e.quc} should be $\frac12$, thus the restriction $\alpha\in (0,\tfrac12)$ is slightly suboptimal.

\smallskip

Our last main result, another direct application of Theorem~\ref{t.analyticity}, asserts without regularity assumptions on the coefficient field~$\a(\cdot)$ beyond measurability that the operator~$-\nabla\cdot\a\nabla$ has no~$L^2$ eigenfunctions near the bottom of its spectrum. Note that the exponential decay condition~\eqref{e.superfluous} below is superfluous and can be replaced by the weaker assumption that~$u\in L^2$ by a result of Kuchment~\cite[Theorem 6.15]{Ku}.

\begin{theorem}[Absence of embedded eigenvalues, bottom of the spectrum]
\label{t.bottom}
\emph{}\\
There exist constants $\lambda_0(d,\Lambda)>0$ and~$C(d,\Lambda)<\infty$ such that, for every pair 
\begin{equation*}
(\lambda,u) \in [0,\lambda_0]\times H^1_{\mathrm{loc}}(\Rd)
\end{equation*}
satisfying
\begin{equation}
-\nabla \cdot \a\nabla u = \lambda u \quad \mbox{in} \ \Rd
\end{equation}
and the growth condition
\begin{equation}
\label{e.superfluous}
\limsup_{r \to \infty} \, 
\exp(Cr)
\left\| u \right\|_{\underline{L}^2(Q_r)}  = 0, 
\end{equation}
we have that $u\equiv 0$ in $\Rd$. 
\end{theorem}

While we are unaware of Theorem~\ref{t.bottom} appearing previously in the literature, it seems to have attained a kind of folklore status. Indeed, Jonathan Goodman demonstrated to us in a private communication that Theorem~\ref{t.bottom} can also be proved using a Bloch wave perturbation argument. It seems that, like the proof of Theorem~\ref{t.bottom}, this argument is also restricted to the bottom of the spectrum. 

\smallskip

None of the results stated in this introduction can be improved without at least an assumption of Lipschitz continuity of the coefficients (with the possible exception of improving $\alpha$ to $\alpha=\frac12$ in the statement of Theorem~\ref{t.quc}). This however begs the question of whether a smoothness assumption on~$\a(\cdot)$ could be used to augment the proof of Theorem~\ref{t.analyticity} by relaxing the restriction~$r\geq Cm$. Unfortunately, we do not expect this to be the case. 

\smallskip

In the next section, we introduce the heterogeneous polynomials, the higher-order correctors and homogenized equation and prove some auxiliary estimates. The proof of Theorem~\ref{t.analyticity} is the focus of Section~\ref{s.analyticity}. In Section~\ref{s.conseq} we demonstrate that the other four results stated above are corollaries of Theorem~\ref{t.analyticity}.

\section{Higher-order correctors and homogenized tensors}
\label{s.correctors}

\subsection{Preliminaries}
\label{ss.prelim}
The standard basis for~$\Rd$ is denoted~$\{ e_1,\ldots,e_d\}$. We denote the set of multi-indices by~$\N^d$. For each $\alpha\in\N^d$, we define the \emph{order of $\alpha$} by~$\left| \alpha \right| = \sum_{i=1}^d \alpha_i$. The factorial of a multiindex~$\alpha$ is~$\alpha!:= \prod_{i=1}^d \alpha_i!$. For each~$\alpha\in\N^d$, we define the monomial $x^\alpha$ and partial derivative operator~$\partial^\alpha$ by 
\begin{equation}
\label{e.pars}
x^\alpha:= \prod_{i=1}^d x_i^{\alpha_i} 
\quad \mbox{and} 
\quad 
\partial^\alpha = \prod_{i=1}^d \partial_{x_i}^{\alpha_i}.
\end{equation}
We denote, for $|\alpha|=m\in\N$, the multinomial coefficient
\begin{equation}
\binom{m}{\alpha} := \frac{m!}{\alpha!}. 
\end{equation}
By a \emph{symmetric tensor of order $m\in\N$}, we mean a mapping 
\begin{equation}
T := \{ \alpha \in\N^d\,:\, |\alpha|=m \} \to \R.
\end{equation}
We may write $T = 
\{T_\alpha\}_{|\alpha|=m}$. 
We let $\mathbb{T}_m$ denote the set of symmetric tensors of order~$m$. 
The gradient~$\nabla u = (\partial_{x_1}u,\ldots,\partial_{x_d}u)$ of a function~$u$ is a tensor of order~one which we may regard as an element of~$\Rd$. If $u\in C^m$, then we may regard $\nabla^m u = ( \partial^\alpha u)_{|\alpha|=m}$ as a symmetric tensor of order~$m$. For each $m\in\N$ and $x\in\Rd$, we also let $x^{\otimes m} \in \mathbb{T}_m$ be the tensor with coordinates 
\begin{equation}
(x^{\otimes m})_\alpha:= x^\alpha, \quad |\alpha| = m. 
\end{equation}
We denote the contraction~$S:T$ of~$S,T\in\mathbb{T}_m$ by
\begin{equation}
\label{e.colon}
(S:T) := 
\sum_{|\alpha|=m} 
\binom{m}{\alpha}
S_\alpha T_\alpha. 
\end{equation}
The multinomial coefficient appears in order to properly count the multiplicities of the multiindices. It makes writing Taylor expansions convenient, as a polynomial~$p\in\P_m$ of degree~$m\in\N$ can be expressed as 
\begin{equation}
p(x) = \sum_{k=0}^m
\frac{1}{k!} \left( \nabla^kp(0) : x^{\otimes k} \right). 
\end{equation}
For each direction $i\in \{1,\ldots,d\}$, we define the forward finite difference operator 
\begin{equation}
D_i u(x) := u(x+e_i) - u(x). 
\end{equation}
In analogy with~\eqref{e.pars}, for each $n\in\N$, we let $D^nu(x) \in \mathbb{T}_n$ be defined by 
\begin{equation}
(D^\alpha u)(x) = \left( D_{1}^{\alpha_1} D_{2}^{\alpha_2} \cdots D_{d}^{\alpha_d} u \right)(x).
\end{equation}
Note that the forward finite difference operators commute and hence $D^\alpha u$ is indeed a symmetric tensor. 

\smallskip

We next discuss the norms we put on the linear space~$\mathbb{T}_m$ of $m$th order tensors. Even though~$\mathbb{T}_m$ is finite dimensional, since we will work with~$m$ very large and need to keep track of dependence in~$m$, this requires some care. We let $| \cdot |$ denote the Euclidean norm
\begin{equation}
\left| T \right| 
:=
\left( T:T \right)^{\frac12}
=
\left( \sum_{|\alpha|=m} 
\binom{m}{\alpha} 
\left| T_\alpha\right|^2 \right)^{\frac12}, \quad T \in \mathbb{T}_m.
\end{equation}
Observe in particular that, for every $x\in\Rd$ and $m\in\N$, 
\begin{equation}
\label{e.xotimesm.size}
\left| x^{\otimes m} \right| = |x|^m. 
\end{equation}

\subsection{Higher-order correctors}
\label{ss.higherorder}

In order to work with the space~$\mathbb{A}_n$, we need to introduce the higher-order periodic correctors which are classical in the theory of periodic homogenization (see for instance~\cite{ABV,BG}) and represent the ``periodic wiggles'' in the heterogeneous polynomials. Our construction of the correctors is minimalistic because here we are interested only in building the space~$\A_m$ (and not, for instance, in quantifying homogenization errors). In particular, we have no interest in the flux correctors or in the nonsymmetric part of the homogenized tensors.

\smallskip

The higher-order correctors and homogenized tensors can be defined in various different ways, and different choices amount to different parametrizations of the space~$\mathbb{A}_n$ (which is what is actually intrinsic). 
Here we take the correctors to be a family~$\{\phi^{(k)}\}_{k\in\N}$ of $\Z^d$-periodic, mean-zero functions with~$\phi^{(k)}$ valued in the set~$\mathbb{T}_k$ of~$k$th order tensors.  They are constructed simultaneously with a sequence~$\{\ahom^{(k)}\}_{k\in\N}$ of constant tensors, with $\ahom^{(k)}\in \mathbb{T}_k$, such that, for every~$n\in\N$ and~$p\in\P_n$, 
\begin{equation}
\label{e.correq_n}
\nabla \cdot \left( \a \nabla \left( \sum_{k=0}^{n} \phi^{(k)} \, \colon \nabla^k p \right) \right) =
\sum_{k=1}^{n} \ahom^{(k)} \, \colon \nabla^{k} p .
\end{equation}
Here we use the abbreviation $\phi^{(0)}=1$ and $\ahom^{(1)} = 0$. In the case~$n=1$, this is the familiar first-order corrector equation which states that, for every affine~$p\in\P_1$,
\begin{equation} \label{e.correq_1}
\nabla \cdot \left( \a \left( \nabla p + \nabla \left( \phi^{(1)} \, \colon \nabla p \right) \right) \right) = 0.
\end{equation}
The tensor $\ahom^{(2)}$, defined by
\begin{equation}
\label{e.ahom7}
\ahom^{(2)} = \ahom := \left\langle \a \left(I_d + \nabla \phi^{(1)} \right) \right\rangle,
\end{equation}
is the homogenized matrix, which satisfies the ellipticity condition
\begin{equation}
\label{e.ahom.ue}
\left| \xi \right|^2 \leq \xi \cdot \ahom \xi
\quad \mbox{and} \quad
\left| \eta \cdot \ahom \xi \right| 
\leq \Lambda \left| \eta \right| \left| \xi \right|,
\quad
\forall \xi,\eta \in\Rd. 
\end{equation}
We remark that, since we have made no assumption of symmetry of the coefficient field~$\a(\cdot)$, the homogenized tensor defined in~\eqref{e.ahom7}, as well as those appearing below in the proof of Lemma~\ref{l.correctorsexist}, are not in general symmetric. However, since in this paper we only use the tensors to define~$\mathscr{A}$ (i.e, $\ahom^{(k)}$ appears only in expressions of the form $\ahom^{(k)}:\nabla^k u$), there is no need to distinguish between these tensors and their symmetric part. We will therefore abuse notation by assuming that all tensors are symmetric (in other words, one should take the symmetric part of any formula defining a tensor).

\smallskip

The next lemma gives us the existence of correctors of arbitrary degree. 

\begin{lemma} 
\label{l.correctorsexist}
For each $k \in \N$, there exist correctors $\phi^{(k)} \in H^1(\R^d / \Z^d; \mathbb{T}_k )$ with zero mean in $Q_1$ and tensors $\ahom^{(k)} \in \mathbb{T}_k$ such that, for every~$n\in \N$ and~$p \in \mathbb{P}_n$, 
\begin{equation} 
\label{e.corr.eq.n}
-\nabla \cdot \left( \a \nabla \left( \sum_{k=0}^{n} \phi^{(k)} \, \colon \nabla^k p \right) \right) 
= - \sum_{k=2}^{n} \ahom^{(k)} \, \colon \nabla^{k} p 
\quad \mbox{in} \ \Rd. 
\end{equation}
Moreover, there exists a constant $C(d,\Lambda)<\infty$ such that, for every $k \in \N$,
\begin{equation}
\label{e.corrector.bounds}
\left\|  \nabla \phi^{(k)} \right\|_{L^2(Q_1)} 
+ \left| \ahom^{(k)} \right| 
\leq C^k .
\end{equation}
\end{lemma}

\begin{proof}
We proceed inductively. Let us assume that we have $\phi^{(k)} \in H^1(\R^d / \Z^d; \mathbb{T}_k)$ and $\ahom^{(k)} \in  \mathbb{T}_k$ for $k \in \{1,\ldots,m-1\}$, such that~\eqref{e.corr.eq.n} is valid for $n \in \{1,\ldots,m-1\}$. We will construct $\phi^{(m)} \in H^1(\R^d / \Z^d;  \mathbb{T}_m )$ and $\ahom^{(m)} \in  \mathbb{T}_m$ such that~\eqref{e.corr.eq.n} is valid for $n=m$. The base case~$m=1$ for the induction is valid since $\ahom^{(1)} = 0$ and $\phi^{(0)}=1$ by the equation of $\phi^{(1)}$ in~\eqref{e.correq_1}. Let us denote, for $n\in\{1,\ldots, m-1\}$, 
\begin{equation}
L_n[p] := 
\nabla \cdot \left( \a \nabla \left( \sum_{k=0}^{n} \phi^{(k)} \, \colon \nabla^k p \right) \right) 
- \sum_{k=2}^{n} \ahom^{(k)} \, \colon \nabla^{k} p.
\end{equation}
Let $p \in \mathbb{P}_m$. Note that, for every $z\in\Rd$, we have that~$\tilde p := p - p(\cdot-z) \in \mathbb{P}_{m-1}$. We set
$p_m := \frac1{m!} \nabla^m p(0) : x^{\otimes m}$ and observe, by the induction assumption and the linearity of $L_{n}$, that $L_{m-1}[p] = L_{m-1}[p_m]$ and 
\begin{align} \notag 
	L_{m-1}[p] - L_{m-1}[p(\cdot-z)] = L_{m-1}[p - p(\cdot-z)] = 0.
\end{align}
Taking $z\in \Z^d$ above and using the periodicity of~$\a(\cdot)$, we deduce that $L_{m-1}[p_m]$ is a periodic distribution. We then define the tensor $\ahom^{(m)}\in \mathbb{T}_m$ by  
\begin{equation*} 
\left( \ahom^{(m)} \, \colon \nabla^m p \right)  := \left\langle L_{m-1} [p_m] \right\rangle,
\quad 
p\in\P_m. 
\end{equation*}
It follows that $L_{m-1}[p] - \ahom^{(m)} \, \colon \nabla^m p$ is periodic, belongs to $H^{-1}(\R^d / \Z^d )$, and has zero mean. Therefore, by Lax-Milgram lemma, the equation 
\begin{equation*} 
	\nabla \cdot \left( \a \nabla u_p \right) =   \ahom^{(m)} \, \colon \nabla^m p_m - L_{m-1}[p_m] 
\end{equation*}
has a solution $u_p\in H^1(\R^d / \Z^d; \mathbb{T}_m )$ with zero mean and satisfying the bound
\begin{equation*} 
	\left\| \nabla u_p  \right\|_{L^2(Q_1)} \leq C \left\| \ahom^{(m)} \, \colon \nabla^m p_m - L_{m-1}[p_m]    \right\|_{H^{-1}(\R^d / \Z^d )}.
\end{equation*}
By the linearity of the map $p\mapsto u_p$ we may write this in tensor form, that is, $u_p = \phi^{(m)} \colon \nabla^m p $ for all $p \in \mathbb{P}_m$. This proves the induction step and completes the proof of the first statement. 
 
\smallskip 
 
The bound~\eqref{e.corrector.bounds} is straightforward to obtain by induction. 
\end{proof}

\subsection{Heterogeneous polynomials: the space $\mathbb{A}_m$}

We define~$\mathscr{A}$ to be the ``full'' higher-order macroscopic (or homogenized) operator
\begin{equation}
\label{e.def.scrA}
\mathscr{A} u := 
\sum_{n=2}^\infty 
\left( \ahom^{(n)} : \nabla^n u \right).
\end{equation}
This operator is a more precise large-scale approximation of the heterogenous operator~$\nabla \cdot \a \nabla$ than the usual homogenized operator~$\nabla \cdot \ahom\nabla$, which is nothing other than the first summand on the right of~\eqref{e.def.scrA}. Unfortunately,~$\mathscr{A}$ is not elliptic.
However, as we will see, in many situations it is dominated by its \emph{lowest}-order coefficients (at least for sufficiently regular data) and thus, in a certain sense,~$\nabla \cdot \ahom\nabla$ is the leading order approximation of~$\mathscr{A}$. Therefore~$\mathscr{A}$ is relatively well-behaved if its domain is restricted to very regular functions. 

\smallskip

For each $m\in\N$, define
\begin{equation}
\mathbb{A}_m:=
\left\{ 
\psi = \sum_{k=0}^m \left( \nabla^k q : \phi^{(k)} \right)
\,:\, 
q\in \P_m
\right\}. 
\end{equation}
By Lemma~\ref{l.correctorsexist}, $\nabla \cdot \a\nabla \psi \in \P_{m-2}$ for every $\psi\in \A_{m}$. Indeed, for every $q\in \P_m$, 
\begin{equation}
\label{e.poly.consistency}
-\nabla \cdot \a\nabla \psi = -\mathscr{A}q 
\quad \mbox{where} \quad 
\psi = \sum_{k=0}^m \nabla^k q: \phi^{(k)}.
\end{equation}
In this sense, the approximation of~$\nabla \cdot \a\nabla$ by~$\mathscr{A}$ is exact, for elements of~$\psi$. The result of Avellaneda and Lin~\cite{AL2} tells us that the set~$\mathbb{A}_n$ precisely characterizes the set of solutions~$u$ of the equation
\begin{equation}
\label{e.pde.encour}
-\nabla \cdot \a\nabla u = p 
\end{equation}
where~$p\in \A_{m-2}$ and $|u(x)|$ grows at most like~$o(|x|^{m+1})$ at infinity. In other words, for every $m\in\N$, 
\begin{align}
\label{e.Liouville.poly}
\left\{ 
u\in H^1_{\mathrm{loc}} (\Rd)
\, :\,  
 -\nabla \cdot \a\nabla u \in\P_{m-2}, \ 
\limsup_{r\to \infty} r^{-(m+1)} \left\| u \right\|_{\underline{L}^2(B_r)} = 0
\right\}
= 
\mathbb{A}_m. 
\end{align}
In particular, the set of~$\a(x)$--harmonic functions (solutions of~\eqref{e.pde.encour} with $p=0$) with at-most polynomial growth is characterized by
\begin{align}
\label{e.Liouville.poly.0}
&
\left\{ 
u\in H^1_{\mathrm{loc}} (\Rd)
\, :\,  
  -\nabla \cdot \a\nabla u = 0, \ 
\limsup_{r\to \infty} r^{-(m+1)} \left\| u \right\|_{\underline{L}^2(B_r)} = 0
\right\}
\\ & \notag \qquad\qquad\qquad\qquad 
= 
\left\{ 
u = \sum_{k=0}^m \left( \nabla^k q : \phi^{(k)} \right)
\,:\, 
q\in \P_m, \ \mathscr{A}q = 0
\right\}  = : \mathbb{A}^0_m.
\end{align}
This Liouville-type theorem for solutions with polynomial growth is a qualitative version of the $C^{m,1}$ estimate~\eqref{e.Cm1}. It is an immediate consequence of the weaker version of~\eqref{e.Cm1} in which the constant~$C$ on the right side is allowed to depend on~$m$. Due to~\eqref{e.Liouville.poly}, we informally refer to elements of~$\A_m$ as ``heterogeneous polynomials.'' We will give another proof of~\eqref{e.Liouville.poly} in this paper, as the argument for Theorem~\ref{t.analyticity} does not depend on it. 

\smallskip

We remark that coefficients in the operator~$\mathscr{A}$ are not universal. Indeed, if one normalized the correctors differently or defined them with respect to the cube $Q_2$ instead of $Q_1$, then all of the homogenized coefficients would in general be different, with the exception of~$\ahom^{(2)}$. The universal objects are the tensor~$\ahom^{(2)}$ (the homogenized matrix) and the spaces~$\mathbb{A}_n$. 

\smallskip

We note that there exists $C(d,\Lambda)<\infty$ such that, for every $m\in\N$, $\psi\in \A_m$ and $r\geq 1$, 
\begin{equation}
\label{e.psi.easybound}
\left\| \psi \right\|_{\underline{L}^2(Q_r)}
\leq
\sum_{k=0}^m
\left( \frac{Cr}{k+1} \right)^k \left| \nabla^kq(0) \right|,
\quad \mbox{where} \quad 
\psi = \sum_{k=0}^m \nabla^k q:\phi^{(k)},\quad q\in\P_m.  
\end{equation}
This is an immediate consequence of~\eqref{e.corrector.bounds}.

\smallskip

In the rest of this section, we establish some properties of the space~$\mathbb{A}_m$ and macroscopic operator~$\mathscr{A}$. We first give a basic result concerning the Laplacian operator restricted to polynomials.

\begin{lemma}
\label{l.polyharmonic}
There exists $C(d,\Lambda)<\infty$ such that, for every~$m\in\N$ and~$p\in \P_m^*$, there exists $q\in \P_{m+2}^*$ satisfying 
\begin{equation}
-\nabla \cdot \ahom\nabla q = p \quad \mbox{in} \ \Rd
\end{equation}
and
\begin{equation}
\label{e.polyharmonic}
\left| \nabla^{m+2}q \right| 
\leq 
C^m \left| \nabla^m p \right|. 
\end{equation}
\end{lemma}
\begin{proof}
By performing an affine change of coordinates, we may suppose that~$\ahom=\Id$, in other words,~$\nabla \cdot \ahom\nabla = \Delta$. We claim that the polynomial $q\in \P_{m+2}^*$ defined by 
\begin{align} \notag 
q(x) := - \sum_{j=0}^\infty a_{j,m} |x|^{2(j+1)} (-\Delta)^{j} p(x)
\end{align}
satisfies the equation
\begin{equation} 
\label{e.Deltapisq}
- \Delta q = p \quad \mbox{in} \ \Rd, 
\end{equation}
where the coefficients $\{ a_{j,m}\}$ are defined, recursively by
\begin{equation*} 
a_{-1,m} = 1, \quad a_{j,m} = \frac{1}{2(j+1)( 2(m-j) + d)} a_{j-1,m} , \quad j \in \N.
\end{equation*}
Notice that we have 
\begin{equation*} 
a_{j,m} 
= 
\prod_{i=0}^j \frac{1}{2(i+1)( 2(m-i) + d)} .
\end{equation*}
To prove~\eqref{e.Deltapisq}, a direct computation gives that
\begin{align} \notag 
\lefteqn{
\Delta \left( |x|^{2(j+1)} (-\Delta)^j p(x) \right)
} \quad &
\\ 
\notag &
= 2(j+1)|x|^{2j} \left( (2j+d)+ 2 x\cdot \nabla \right) (-\Delta)^j p(x) - |x|^{2(j+1)} (-\Delta)^{j+1} p(x) 
\\ 
\notag &
= 2(j+1) \left( 2(m - j) + d \right) |x|^{2j} (-\Delta)^j p(x) - |x|^{2(j+1)} (-\Delta)^{j+1} p(x) 
. 
\end{align}
By the definition of $a_{j,m}$, we have that
\begin{equation*} 
2(j+1) \left( 2(m - j) + d\right) a_{j,m} = a_{j-1,m} 
\end{equation*}
and, therefore,
\begin{align} \label{e.Lap.telescoping}
\Delta \left( a_{j,m} |x|^{2(j+1)}(-\Delta)^{j} p(x) \right) 
& 
= 
a_{j-1,m} |x|^{2j} (-\Delta)^{j} p(x)
\\ \notag
&
\quad 
- a_{j,m} |x|^{2(j+1)} (-\Delta)^{j+1} p(x) .
\end{align}
Thus,~\eqref{e.Deltapisq} follows by telescoping since $a_{-1,m} = 1$. 

\smallskip
Next, taking gradients, we see that 
\begin{align*} 
\nabla^{m+2} q 
& 
= \sum_{j=0}^\infty a_{j,m} \sum_{k=0}^{m+2} \binom{m+2}{k} \left( \nabla^{k} |x|^{2(j+1)} \otimes \nabla^{m+2-k} (-\Delta)^{j} p(x) \right)(0)
\\ &
= 
\sum_{j=0}^{\lfloor \frac m2 \rfloor} a_{j,m} \binom{m+2}{2(j+1)} (2(j+1))! I_d^{\otimes (j+1)} \otimes \nabla^{m - 2j} (-\Delta)^{j} p 
\\ &
= 
\sum_{j=0}^{\lfloor \frac m2 \rfloor} (-1)^j a_{j,m} \frac{(m+2)!}{(m-2j)!} I_d^{\otimes (j+1)} \otimes \nabla^{m - 2j} (\Delta)^{j} p 
.
\end{align*}
The tensor $\nabla^{m - 2j} (-\Delta)^{j}$ can be estimated as
\begin{equation*} 
\left| I_d^{\otimes (j+1)} \otimes \nabla^{m - 2j} (-\Delta)^{j} p \right| \leq d^{\frac{j}2} |\nabla^m p| ,
\end{equation*}
and it follows that
\begin{equation*} 
\left| \nabla^{m+2} q \right| \leq |\nabla^m p| \sum_{j=0}^{\lfloor \frac m2 \rfloor} a_{j,m} C^j \frac{(m+2)!}{(m-2j)!}  .
\end{equation*}
Now, we estimate
\begin{equation*} 
\sum_{j=0}^{\lfloor \frac m2 \rfloor} a_{j,m}C^j
\frac{(m+2)!}{(m-2j)!} 
= 
\sum_{j=0}^{\lfloor \frac m2 \rfloor} 
a_{j,m} (Cj)^j \binom{m+2}{2j}
\leq 
C^m,
\end{equation*}
using
\begin{equation*} 
(2j)! a_{j,m} = \frac1{2j(2m+d)} \prod_{i=0}^{j-1} \frac{ 2(i+1) (2i + 1)}{2(i+1)( 2(m-j) + 2 i + d)} \leq \frac1{2j(2m+d)} .
\end{equation*}
This proves~\eqref{e.polyharmonic}. 
\end{proof}

The following lemma asserts that the macroscopic operator~$\mathscr{A}$ may be inverted on the space of polynomials. 

\begin{lemma}
\label{l.twist}
There exists $C(d,\Lambda)<\infty$ such that, for every~$m\in\N$ and~$p\in \P_m$, there exists $q\in \P_{m+2}$ satisfying 
\begin{equation}
-\mathscr{A} q = p \quad \mbox{in} \ \Rd
\end{equation}
such that $q(0)=0$, $\nabla q(0)=0$ and, for every $n\in\{0,\ldots,m\}$, 
\begin{equation}
\label{e.twist.bounds}
\left| \nabla^{n+2}q(0) \right|
\leq 
\sum_{k=n}^m 
C^{k} \left| \nabla^k p(0) \right|. 
\end{equation}
\end{lemma}
\begin{proof}
By Lemma~\ref{l.polyharmonic}, there exists $q_0\in \P_{m+2}$ satisfying
\begin{equation}
\left\{
\begin{aligned}
& 
-\nabla \cdot \ahom \nabla q_0 = 
p
\\ & 
q_0(0)= 0, \ \nabla q_0(0)=0, 
\\ &
\left| \nabla^{n} q_0(0) \right| 
\leq C^n \left|\nabla^{n-2} p(0) \right|, \ \ \forall  n\in\{2,\ldots,m+2\}. 
\end{aligned}
\right.
\end{equation}
Arguing inductively using Lemma~\ref{l.polyharmonic} and~\eqref{e.scrA.forward}, there exist~$q_k \in \P_{m+2-k}$ satisfying, for every $k\in\{ 1,\ldots,m+1\}$,
\begin{equation}
\left\{
\begin{aligned}
& 
-\nabla \cdot \ahom \nabla q_{k} = 
\mathscr{A}q_{k-1} - \nabla \cdot \ahom \nabla q_{k-1},
\\ & 
q_k(0)= 0, \ \nabla q_k(0)=0, 
\\ &
\left| \nabla^{n} q_0(0) \right| 
\leq
C^{n+k} \left|\nabla^{n-2+k} p(0) \right|, \quad \forall  n\in\{2,\ldots,m+2-k\}.
\end{aligned}
\right.
\end{equation}
Note that $q_{m+1}=0$. 
Thus, setting $q:= \sum_{k=0}^{m} q_k$, we obtain the lemma.
\end{proof}

We next compare the kernel of~$\mathscr{A}$ to that of~$\nabla \cdot \ahom\nabla$ and show that they agree at leading order.

\begin{lemma}
\label{l.Akernel}
There exists $C(d,\Lambda)<\infty$ such that, for every $m\in\N$ and $p\in \P_m$ satisfying~$-\nabla \cdot \ahom\nabla p = 0$, there exists $p' \in \P_{m}$ satisfying
\begin{equation}
-\mathscr{A}p' = 0,
\end{equation}
and
\begin{equation}
p'(0)=p(0), \quad \nabla p'(0) = \nabla p(0)
\quad \mbox{and} \quad 
\nabla^m p' = \nabla^m p 
\end{equation}
such that,  for every $n\in\{2,\ldots,m-1\}$, 
\begin{equation}
\label{e.twist.it}
\left| \nabla^{n}(p-p')(0) \right| 
\leq 
\sum_{k=n+1}^{m} C^{k} \left| \nabla^k p(0) \right|.
\end{equation}
In particular, for every $r > Cm$, 
\begin{equation}
\label{e.twist.it.far}
\left\| p - p' \right\|_{L^\infty(Q_r)} 
\leq
\sum_{k=3}^{m}
\left( \frac{Cr}{k} \right)^{k-1} \left| \nabla^k p(0) \right|.
\end{equation}
\end{lemma}
\begin{proof}
We may suppose $m\geq 3$, since $\mathscr{A}$ and $\nabla \cdot \a\nabla$ coincide on~$\P_{2}$. 
Let~$p\in \P_m$ be $\ahom$-harmonic. 
Then $\mathscr{A} p \in \P_{m-3}$ and thus, by applying Lemma~\ref{l.twist} and using~\eqref{e.scrA.forward}, we can find $q \in \P_{m-1}^*$ satisfying
\begin{equation}
-\mathscr{A}q = -\mathscr{A}p \quad \mbox{in} \ \Rd
\end{equation}
such that $q(0)=0$, $\nabla q(0)= 0$ and, for every $n\in\{3,\ldots,m\}$,  
\begin{align}
\label{e.twisties}
\left| \nabla^{n-1}q(0) \right| 
&
\leq 
\sum_{k=n-3}^{m-3} 
C^{k} \left| \nabla^k \mathscr{A} p(0) \right|
\\ & \notag
\leq
\sum_{k=n-3}^{m-3} 
C^{k}
\sum_{j=k+3}^m
C^{j} \left| \nabla^j p(0) \right| 
\\ & \notag
\leq
\sum_{j=n}^{m}
\sum_{k=n-3}^{j-3}
C^j \left| \nabla^j p(0) \right|
\leq 
\sum_{j=n}^{m} C^{j} \left| \nabla^j p(0) \right|.
\end{align}
Setting $p':= p-q$ yields~\eqref{e.twist.it}.

\smallskip

To get~\eqref{e.twist.it.far}, we estimate, for $|x|>Cm$, 
\begin{align*}
\left| p(x) - p'(x) \right| 
&
=
\left| 
\sum_{n=2}^{m-1} 
\nabla^n(p-p')(0) : \frac{x^{\otimes n}}{n!}
\right| 
\\ & 
\leq
\sum_{n=2}^{m-1} 
\left( \frac{C |x|}{n} \right)^n
\left| \nabla^n(p-p')(0) \right|
\\ & 
\leq 
\sum_{n=2}^{m-1} 
\left( \frac{C |x|}{n} \right)^n
\sum_{j=n+1}^{m} C^{j} \left| \nabla^j p(0) \right|
\\ & 
\leq 
\sum_{j=3}^{m}
\sum_{n=2}^{j-1} 
\left( \frac{C |x|}{n} \right)^n
C^{j} \left| \nabla^j p(0) \right|
\\ & 
\leq 
\sum_{j=3}^{m}
\left( \frac{C|x|}{j} \right)^{j-1} \left| \nabla^j p(0) \right|.
\end{align*}
This completes the proof. 
\end{proof}

\begin{lemma}
\label{l.polyhit.bound}
There exists $C(d,\Lambda)<\infty$ such that, for every $m\in\N$ and~$p\in\P_m$, there exists $\psi \in \mathbb{A}_{m+2}$ satisfying 
\begin{equation}
\label{e.polyhit}
\left\{ 
\begin{aligned}
& -\nabla \cdot \a\nabla \psi = p \quad \mbox{in} \ \Rd
\\ & 
\hat{\psi}(0) = 0, \ D\hat{\psi}(0) = 0,
\end{aligned}
\right. 
\end{equation}
and such that, for every $r\in [m,\infty)$,
\begin{equation}
\label{e.polyhit.bounds}
\left\| \psi \right\|_{\underline{L}^2(Q_r)}
\leq
\sum_{n=0}^m \left( \frac{Cr}{n+1} \right)^{n+2} \left| \nabla^np(0) \right|.
\end{equation}
\end{lemma}
\begin{proof}
By Lemma~\ref{l.twist}, we can find~$q\in \P_{m+2}$ satisfying $-\mathscr{A}q = p$, $\hat{q}(0)=0$, $D\hat{q}(0)=0$ and, for every $n\in\{0,\ldots,m\}$, 
\begin{equation}
\label{e.tw2}
\left| \nabla^{k+2}q(0) \right| 
\leq 
\sum_{n=k}^m 
C^{n} \left| \nabla^n p(0) \right|. 
\end{equation}
Define $\psi:= \sum_{k=0}^{m+2} \nabla^k q:\phi^{(k)}$. 
It is clear from~\eqref{e.poly.consistency} that~$\psi$ satisfies~\eqref{e.polyhit}. By~\eqref{e.psi.easybound} and~\eqref{e.tw2}, for every $r\in [m,\infty)$, 
\begin{align*}
\left\| \psi \right\|_{\underline{L}^2(Q_r)}
&
\leq
\sum_{k=2}^{m+2}
\left( \frac{Cr}{k+1} \right)^k \left| \nabla^kq(0) \right|
\\ & 
\leq
\sum_{k=0}^m 
\sum_{n=k}^m 
\left( \frac{Cr}{k+1} \right)^{k+2} 
C^{n-k} \left| \nabla^n p(0) \right|
\\ & 
= 
\sum_{n=0}^m
\sum_{k=0}^n
\left( \frac{Cr}{k+1} \right)^{k+2} 
C^{n-k} \left| \nabla^n p(0) \right|
\\ & 
\leq 
\sum_{n=0}^m
\left( \frac{Cr}{n+1} \right)^{n+2}
\left| \nabla^n p(0) \right|.
\end{align*}
The proof is complete. 
\end{proof}

\begin{remark}
Notice that we can also put the right hand side of the estimate~\eqref{e.polyhit.bounds} in terms of the differences~$D^np(0)$ rather than the pointwise derivatives~$\nabla^np(0)$. That is, we can replace~\eqref{e.polyhit.bounds} by 
\begin{equation}
\label{e.polyhit.bounds.D}
\left\|
\psi 
\right\|_{\underline{L}^2(Q_r)}
\leq
C \sum_{n=0}^m \left( \frac{Cr}{n+1} \right)^{n+2} \left| D^np(0) \right|.
\end{equation}
To see this, we first notice that, for any polynomial $p\in\P_m$, 
\begin{align}
\label{e.poly.nab.D}
\left| \nabla^n p(0) \right| 
&
\leq
\sum_{k=n}^m
\frac{C^{k-n}}{(k-n)!} \left|D^{k} p(0) \right| k^{k-n}
\leq
\sum_{k=n}^m
\left( \frac{Ck}{k-n} \right)^{k-n}\left|D^{k} p(0) \right|.
\end{align}
Thus, for every $\psi$ and $p$ as in the statement of Lemma~\ref{l.polyhit.bound} and~$r\in [Cm,\infty)$, we may combine~\eqref{e.polyhit.bounds} and~\eqref{e.poly.nab.D} to obtain
\begin{align}
\label{e.polyhit.bounds2}
\left\|
\psi 
\right\|_{\underline{L}^2(Q_r)}
&
\leq
C \sum_{n=0}^m \left( \frac{Cr}{n+1} \right)^{n+2} \left| \nabla^np(0) \right| 
\\ & \notag
\leq
C \sum_{n=0}^m \sum_{k=n}^m
\left( \frac{Cr}{n+1} \right)^{n+2}
\left( \frac{Ck}{k-n} \right)^{k-n}\left|D^{k} p(0) \right| 
\\ & \notag
=
C 
\sum_{k=0}^m
\left( \frac{Cr}{k+1} \right)^{k+2}
\left|D^{k} p(0) \right|
\sum_{n=0}^k
\left( \frac{Ckn}{r(k-n)} \right)^{k-n} 
\\ & \notag
\leq
C \sum_{k=0}^m
\left( \frac{Cr}{k+1} \right)^{k+2}
\left|D^{k} p(0) \right|.
\end{align}
This is~\eqref{e.polyhit.bounds.D}. 
\end{remark}

For $u \in L^1(U)$ and $z\in \Zd$ such that $z+Q_1\subseteq U\subseteq\Rd$, we define
\begin{equation}
\hat{u}(x) : = \int_{z+Q_1} u(x)\,dx \quad \mbox{for } \, x\in z +Q_1.
\end{equation}
We next estimate the growth of an element~$\psi\in \A_m$ by its \emph{intrinsic} differences at the origin, namely $\{ D^k\hat{\psi}(0)\}_{k=0,\ldots,m}$, rather than the differences (or derivatives) of the polynomial~$q$ in its representation, as for instance in~\eqref{e.psi.easybound}. 

\smallskip

\begin{lemma}
\label{l.corr.growth}
There exists $C(d,\Lambda)<\infty$ such that, for every $m,n\in\N$ with $n\leq m$, $q\in \P_m$ and $\psi\in \A_m$ satisfying
\begin{equation}
\label{e.qpsi.relate}
\psi = \sum_{k=0}^m \nabla^kq:\phi^{(k)},
\end{equation}
we have the estimate
\begin{equation}
\label{e.corr.growth.q}
\left( \frac{m}{n+1} \right)^n
\left| \nabla^nq(0) \right|
\leq
\sum_{k=n}^m
\left( \frac{Cm}{k+1} \right)^{k} 
\left| D^k \hat{\psi} (0) \right| 
\end{equation}
and, consequently, for every $r\geq m$, 
\begin{equation}
\label{e.corr.growth}
\left\| \psi \right\|_{\underline{L}^2(Q_r)}
\leq
\sum_{k=0}^m
\left( \frac{Cr}{k+1} \right)^k
\left| D^k \hat{\psi} (0) \right|.
\end{equation}
Moreover, for symmetric tensors $M^{(0)},\ldots,M^{(m)}$ with $M^{(k)} \in (\Rd)^{\otimes k}$, there exists a unique $\psi\in \A_m$ such that 
\begin{equation}
\label{e.slap.ya}
D^k \hat{\psi} (0) = M^{(k)}, 
\qquad \forall k\in\{0,\ldots,m\}.
\end{equation}
\end{lemma}
\begin{proof}
We proceed by first proving the estimate~\eqref{e.corr.growth.q} by induction in~$m$. The statement is clearly true for $m\in\{0,1\}$. Suppose, for some $m\in\N$ with $m\geq 1$, that~\eqref{e.corr.growth} is valid for~$m-1$ in place of~$m$, that is, there is $\mathsf{C}(d,\Lambda)<\infty$ such that, for every $\psi \in \A_{m-1}$ and~$q\in \P_{m-1}$ satisfying~\eqref{e.qpsi.relate} and $n\in\{0,\ldots,m-1\}$, we have 
\begin{equation}
\label{e.corr.growth.ind.q}
\left( \frac{m-1}{n+1} \right)^n
\left| \nabla^nq(0) \right|
\leq
\sum_{k=n}^{m-1}
\left( \frac{ \mathsf{C} (m-1)}{k+1} \right)^{k}
\left| D^k \hat{\psi} (0) \right|.
\end{equation}
Let $\psi \in \A_{m}$ and~$q\in \P_{m}$ be as in the lemma. Let $q_m$ be the unique polynomial of degree~$m$ satisfying
\begin{equation}
D^mq_m = D^m\psi  
\quad \mbox{and} \quad
D^kq_m(0) = 0,  \ \  \forall k\in\{0,\ldots,m-1\},
\end{equation}
and let~$\zeta \in \A_{m}$ be defined by
\begin{equation}
\zeta := 
\sum_{k=0}^{m}  \nabla^k q_m :  \phi^{(k)}. 
\end{equation}
Then, for every $k,n\in\{0,\ldots,m\}$ with $k+n\leq m$ and $x\in\Rd$, 
\begin{equation}
\label{e.nabkDn.bound}
\left| \nabla^k D^n q_m(x) \right| 
\leq
\frac{C^{m-n-k}}{(m-n-k)!} \left|D^{m} \psi \right| (k+|x|)^{m-n-k}.
\end{equation}
Hence, by~\eqref{e.corrector.bounds} and~\eqref{e.psi.easybound}, we get, for every $r\geq m$ and~$n\in\{0,\ldots,m\}$, 
\begin{align}
\label{e.zeta.bound}
\left| D^n \hat\zeta(0) \right| 
& \leq
\left\| D^n \zeta \right\|_{\underline{L}^2(Q_1)}
\\ & \notag
\leq
\sum_{k=0}^{m-n} \left\|  \nabla^k D^nq_m  \, \colon \phi^{(k)}  \right\|_{\underline{L}^2(Q_1)}
\\ & \notag
\leq
\left|D^{m}\psi \right| 
\sum_{k=0}^{m-n}
C^k
\left(
\frac{C(k+1)}{m-n-k}
\right)^{m-n-k}
\\ & \notag
\leq 
C^{m-n}
\left|D^{m} \psi \right|. 
\end{align}
Since $\psi - \zeta \in \A_{m-1}$, by the induction hypothesis~\eqref{e.corr.growth.ind.q},~\eqref{e.nabkDn.bound} and~\eqref{e.zeta.bound}, we obtain, for every~$n\in\{0,\ldots,m-1\}$,
\begin{align*}
\lefteqn{
\left( \frac{m}{n+1} \right)^n\left| \nabla^n q(0) \right| 
} \quad & 
\\ &
\leq 
\left( \frac{m}{n+1} \right)^n\left| \nabla^n (q-q_m)(0) \right| 
+ \left( \frac{m}{n+1} \right)^n\left| \nabla^n q_m(0) \right|
\\ & 
\leq 
\left( \frac{m}{m-1} \right)^n\,
\sum_{k=n}^{m-1}
\left( \frac{ \mathsf{C} (m-1)}{k+1} \right)^{k}
\left(
\left| D^k \hat{\psi}(0) \right| + \left| D^k \hat{\zeta}(0) \right|\right)
+
\left( \frac{Cn}{m-n} \right)^{m-n} 
\left| D^m \psi \right|
\\ & 
\leq 
\sum_{k=n}^{m-1}
\left( \frac{ \mathsf{C}m}{k+1} \right)^{k}
\left| D^k \hat{\psi}(0) \right|
+
\left| D^m \psi \right|
\sum_{k=n}^{m-1} C^{m-k} \left( \frac{ \mathsf{C}m}{k+1} \right)^{k}
+
C^m  \left| D^m \psi \right|.
\end{align*}
If~$\mathsf{C}$ is sufficiently large ($\mathsf{C} \geq 8 e C$ suffices\footnote{Details can be found in a commented-out portion of the latex source of this paper, which can be downloaded from the arXiv.}), 
then the second term on the right side is bounded by
\begin{align*}
\sum_{k=n}^{m-1} C^{m-k} \left( \frac{ \mathsf{C}m}{k+1} \right)^{k}
\leq 
\frac12\left( \frac{\mathsf{C}m}{m+1} \right)^{m}.
\end{align*}
Combining the above and requiring $\mathsf{C}$ to be sufficiently large once again, we obtain that, for every~$n\in\{0,\ldots,m-1\}$, 
\begin{equation}
\label{e.indystepcomplete}
\left( \frac{m}{n+1} \right)^n
\left| \nabla^n q(0) \right| 
\leq
\sum_{k=n}^{m}
\left( \frac{ \mathsf{C} m}{k+1} \right)^{k}
\left| D^k \hat{\psi} (0) \right|.
\end{equation}
Note that the bound~\eqref{e.indystepcomplete} for $n=m$ was proved already in~\eqref{e.zeta.bound}. 
This completes the induction step and thus the proof of~\eqref{e.corr.growth.q}. 

\smallskip

The second estimate~\eqref{e.corr.growth} is an immediate consequence of~\eqref{e.psi.easybound} and~\eqref{e.corr.growth.q}. Indeed, combining these yields, for every $r\geq m$, 
\begin{align*}
\left\| \psi \right\|_{\underline{L}^2(Q_r)}
&
\leq 
\sum_{k=0}^m
\left( \frac{Cr}{k+1} \right)^k \left| \nabla^k q(0) \right|
\\ & 
\leq
\sum_{k=0}^m \sum_{n=k}^m
\left( \frac{Cr}{m} \right)^k
\left( \frac{Cm}{n+1} \right)^n
\left| D^n \hat{\psi}(0) \right| 
\\ & 
=
\sum_{n=0}^m \sum_{k=0}^n
\left( \frac{Cr}{m} \right)^k
\left( \frac{Cm}{n+1} \right)^n
\left| D^n \hat{\psi}(0) \right| 
\\ & 
\leq 
\sum_{n=0}^m
\left( \frac{Cr}{m} \right)^n
\left( \frac{Cm}{n+1} \right)^n
\left| D^n \hat{\psi}(0) \right| 
=
\sum_{n=0}^m 
\left( \frac{Cr}{n+1} \right)^n 
\left| D^n \hat{\psi}(0) \right| .
\end{align*}

\smallskip

Turning to the proof of the last statement, we note that uniqueness is clear from the fact that $\psi \in \A_{m}$ and $D^m\hat{\psi}(0)=0$ implies that $\psi \in \A_{m-1}$. We prove the existence part of the second statement by induction. Its validity is clear for~$m\in\{0,1\}.$ We therefore suppose, for some~$m\in\N$ with~$m\geq 1$, the statement is valid for~$m-1$ in place of~$m$. Let $M^{(0)},\ldots,M^{(m)}$ be as in the statement and define~$q\in \P_{m}^*$ and~$\zeta \in \A_{m}$ by
\begin{equation}
q(x):= \frac1{m!} M^{(m)} : x^{\otimes m} 
\quad \mbox{and} \quad
\zeta(x):= 
\sum_{k=0}^{m}  \nabla^k q(x) :  \phi^{(k)}(x). 
\end{equation}
It is clear that~$D^m\zeta = M^{(m)}$. 
By the induction hypothesis, there exists~$\xi\in \A_{m-1}$ such that 
\begin{equation}
D^k \hat{\xi}(0) = M^{(k)} - D^k \hat{\zeta}(0), \quad \forall k\in\{0,\ldots,m-1\}. 
\end{equation}
Defining $\psi:= \zeta + \xi \in \A_{m}$, we obtain~\eqref{e.slap.ya}. The proof of the second statement, and thus of the lemma, is now complete.
\end{proof}

\subsection{Entire solutions with slow exponential growth: the space $\A_\infty(\delta)$}

Recall that, for each $\delta>0$, we denote \begin{equation}
\label{e.Pinftydelta}
\mathbb{P}_\infty(\delta) 
:= 
\left\{ 
w \, : \,
w = \sum_{n=0}^\infty w_n, \ w_n \in \mathbb{P}_n^*, \
\sum_{n=0}^\infty 
\delta^{-n} \left| \nabla^n w_n(0) \right| 
< \infty
\right\}.
\end{equation}
The linear space $\P_\infty(\delta)$ is a Banach space with respect to the norm
\begin{equation}
\left\| u \right\|_{\P_\infty(\delta)} 
:=
\sum_{n=0}^\infty 
\delta^{-n} \left| \nabla^n u(0) \right|.
\end{equation}
By~\eqref{e.corrector.bounds}, there exists $\delta_0(d,\Lambda)>0$ such that~$\mathscr{A}u$ is well-defined for every $u\in \P_\infty(\delta_0)$ and moreover maps~$\mathbb{P}_\infty(\delta)$ to itself for every $\delta \in (0,\delta_0]$. Indeed, if $\delta$ is sufficiently small, then, for every $u\in \P_{\infty}(\delta)$ and $k\in\N$, 
\begin{align}
\label{e.scrA.forward}
\left\| \mathscr{A} u \right\|_{\P_\infty(\delta)}
&
=
\sum_{n=0}^\infty
\delta^{-n} \left| \nabla^n \mathscr{A} u(0) \right|
\\ & \notag
\leq
\sum_{n=0}^\infty
\sum_{k=n+2}^\infty
\delta^{-n} C^{k} \left| \nabla^k u(0) \right| 
\\ & \notag
\leq 
\sum_{k=2}^\infty
\sum_{n=0}^{k-2}
(\delta C)^{k}
\delta^{-k} \left| \nabla^k u(0) \right| 
\leq
C \delta^{2} \left\| u \right\|_{\P_\infty(\delta)} .
\end{align}
We say that $u\in \P_\infty(\delta)$ is \emph{$\mathscr{A}$--harmonic} if $\mathscr{A} u = 0$. 

\smallskip

For $\delta \in (0,\delta_0]$, we define the space 
\begin{equation}
\A_\infty(\delta):=
\left\{ 
\psi = \sum_{k=0}^\infty \nabla^k u:\phi^{(k)} 
\, :\,
u \in \P_\infty(\delta) 
\right\},
\end{equation}
where the $\delta_0(d,\Lambda)>0$ is taken small enough that the sum is absolutely convergent, locally uniformly in~$\Rd$. Indeed, observe that if~$\delta_0(d,\Lambda)>0$ is sufficiently small and $\delta\in (0,\delta_0]$, then, for every $u\in \P_\infty(\delta)$,  
\begin{align*}
\left\| 
\sum_{k=0}^\infty \nabla^k u:\phi^{(k)} 
\right\|_{\underline{L}^2(Q_r)}
& 
\leq
\sum_{k=0}^\infty
C^k
\left\| \nabla^k u \right\|_{L^\infty(Q_r)}
\\ & 
\leq 
\sum_{k=0}^\infty
C^k
\sum_{n=k}^\infty 
\left( \frac{Cr}{n-k+1} \right)^{n-k} \left| \nabla^nu(0) \right| 
\\ & 
\leq 
\left\| u \right\|_{\P_\infty(\delta)}
\sum_{k=0}^\infty
(C\delta)^k
\sum_{n=k}^\infty 
\left( \frac{C\delta r}{n-k+1} \right)^{n-k}
\\ & 
\leq 
\left\| u \right\|_{\P_\infty(\delta)}
\exp\left( C\delta r \right)
\sum_{k=0}^\infty
(C\delta)^k
\\ & 
\leq
\left\| u \right\|_{\P_\infty(\delta)}
\exp\left( C\delta r \right).
\end{align*}
In particular, an element $\psi\in \A_\infty(\delta)$ grows at most like a slow exponential:
\begin{equation}
\limsup_{r\to \infty} \,
\exp(-C\delta r)
\left\| \psi \right\|_{\underline{L}^2(Q_r)}
< \infty .
\end{equation}

\smallskip

\section{Large-scale analyticity}
\label{s.analyticity}

In this section we prove the main result of the paper, namely the quantitative, large-scale analyticity of solutions (Theorem~\ref{t.analyticity}).
The first step is to \emph{control the low frequencies} by obtaining regularity of solutions restricted to the lattice~$\Zd$. This is accomplished in Lemma~\ref{l.diffbound}, below, the proof of which is based on the simple idea of Moser and Struwe~\cite{MS} of exploiting the periodic structure of the equation to obtain estimates on integer finite differences of solutions which are analogous to the classical pointwise bounds for the derivatives of harmonic functions. The second step of obtaining control the \emph{high frequencies} is more involved and the focus of most of the section.

\begin{lemma}[Bounds for iterated differences]
\label{l.diffbound}
There exists $C(d,\Lambda) <\infty$ such that, for every 
$m\in\N$, $R \geq m+2$ and $u \in H^1(Q_R)$ satisfying
\begin{equation}
\label{e.pdediffs}
-\nabla \cdot \a\nabla u = 0\quad \mbox{in} \ Q_R,
\end{equation}
we have the estimate
\begin{equation}
\label{e.diffbound}
\left\| D^m u \right\|_{L^2(Q_1)}
\leq
\left( \frac{Cm}{R} \right)^m 
\left\| u \right\|_{\underline{L}^2(Q_R)}.
\end{equation}
\end{lemma}
\begin{proof}
The argument is essentially identical to the one for harmonic functions, as presented for instance in~\cite[Section 2.2.3.c]{Evans}. We first prove the case $m=1$. 
Periodicity implies that~$Du$ is also a solution of~\eqref{e.pdediffs}, and therefore we may apply 
the Jensen and Caccioppoli inequalities to get, for $R>2$ and $r \in [1, R-2)$, 
\begin{equation} \label{e.caccforDu}
\left\| Du \right\|_{\underline{L}^2(Q_{r})}
\leq
\left\| \nabla u \right\|_{\underline{L}^2(Q_{r+2})}
\leq
\frac{C}{R-r-2}
\left\| u \right\|_{\underline{L}^2(Q_{R})},
\end{equation}
and then the De Giorgi-Nash $L^\infty$--$L^2$ estimate to obtain
\begin{equation*}
\left\| Du \right\|_{L^2(Q_{1})}
\leq
\left\| Du \right\|_{L^\infty(Q_{R/2})}
\leq
C \left\| Du \right\|_{\underline{L}^2(Q_{3R/4})} \leq  \frac{C}{R}
\left\| u \right\|_{\underline{L}^2(Q_{R})}.
\end{equation*} 
This gives the result in the case $m=1$. 

\smallskip

We argue by induction to obtain the result for general~$m$. Without loss of generality, we may assume that $R \geq 4m$ since otherwise the result follows simply by the triangle inequality. Assuming the statement is true for~$m$ with constant $C_0$ and again using the fact that $Du$ is a solution, we apply the Caccioppoli inequality~\eqref{e.caccforDu} to obtain, for every~$R \geq 4m$, 
\begin{align*}
\left\| D^{m+1} u \right\|_{L^2(Q_{1})}
&
\leq 
\left( \frac{C_0 (m+2)}{R} \right)^m
\left\| D u \right\|_{\underline{L}^2(Q_{(1-\frac{2}{m+2})R})}
\\ &
\leq 
\left( \frac{C_0(m+2)}{R} \right)^m
\frac{C(m+2)}{R} \left\| u \right\|_{\underline{L}^2(Q_{R})}
\\ & 
\leq 
\left( \frac{C^{\frac1{m+1}} C_0^{\frac{m}{m+1}} (m+1)}{R} \right)^{m+1}
\left\| u \right\|_{\underline{L}^2(Q_{R})}.
\end{align*}
If $C_0 \geq C$ is sufficiently large, we obtain the statement for $m+1$. 
\end{proof}

We next present a simple but useful lemma which estimates the $L^2$ norm of a function in a cube~$Q_r$ by the values of $\hat{u}$ on the lattice $Q_r\cap \Zd$ and the $L^2$ norm of its gradient. This ``small-scale Poincar\'e inequality'' is the basic tool we use to control the high frequencies of the solutions. 

\begin{lemma}
\label{l.jumpdown}
There exists $C(d)<\infty$ such that, for every~$r\in\N$ with $r\geq 1$ and~$u\in H^1(Q_r)$,
\begin{equation}
\label{e.jumpdown}
\left\| u \right\|_{\underline{L}^2(Q_r)} 
\leq 
\left( \frac{1}{|Q_r|} \sum_{z\in \Zd\cap Q_r} 
\left| \hat{u}(z) \right|^2 \right)^{\frac12}
+
C 
\left\| \nabla u\right\|_{\underline{L}^2(Q_r)}.
\end{equation}
\end{lemma}
\begin{proof}
Using the triangle inequality and the Poincar\'e inequality, we compute
\begin{align}
\label{e.jumpdown.p}
\left\| u \right\|_{\underline{L}^2(Q_r)} 
& 
=
\left( \frac{1}{|Q_r|} \sum_{z\in \Zd\cap Q_r} 
\left\| u \right\|_{{L}^2(z+Q_1)}^2
\right)^{\frac12}
\\ & \notag
\leq 
\left( \frac{1}{|Q_r|} \sum_{z\in \Zd\cap Q_r} 
\left( \left| \hat{u}(z) \right|^2 
+
\left\| u - \hat{u}(z) \right\|_{{L}^2(z+Q_1)}^2\right)
\right)^{\frac12} 
\\ & \notag
\leq 
\left( \frac{1}{|Q_r|} \sum_{z\in \Zd\cap Q_r} 
\left|\hat{u}(z) \right|^2 
\right)^{\frac12} 
+
C \left( 
\frac{1}{|Q_r|} \sum_{z\in \Zd\cap Q_r} 
\left\| \nabla u \right\|_{{L}^2(z+Q_1)}^2
\right)^{\frac12} 
\\ & \notag
= 
\left( \frac{1}{|Q_r|} \sum_{z\in \Zd\cap Q_r} 
\left| \hat{u}(z) \right|^2 
\right)^{\frac12} 
+
C
\left\| \nabla u \right\|_{\underline{L}^2(Q_r)}.
\qedhere
\end{align}
\end{proof}

\begin{remark}
While the De Giorgi-Nash estimate was used in the proof of Lemma~\ref{l.diffbound}, one can avoid this (and thus have an argument which works for systems) by instead relying on Lemma~\ref{l.jumpdown}. One proceeds by first iterating~\eqref{e.caccforDu} many times (depending on $d$) and then applying a version of the Sobolev inequality on~$\Zd$ to get a uniform  estimate on~$D\hat{u}$. Then Lemma~\ref{l.jumpdown} can be invoked to take care of the small scales. 
\end{remark}

For expository purposes, and as a warm-up to the proof of Theorem~\ref{t.analyticity}, we next give a simple new proof of the large-scale~$C^{0,1}$ and $C^{1,1}$ estimates of Avellaneda-Lin~\cite{AL1}. All previous proofs rely on $\ahom$--harmonic approximation via homogenization, either via a compactness argument~\cite{AL1} or a quantitative homogenization argument~\cite{AS}. Here we give a more direct and elementary argument which makes no (explicit) use of homogenization or $\ahom$--harmonic approximation, and yet is completely quantitative. The proof uses only the previous two lemmas and the existence of first-order correctors. 

\begin{lemma}[Large-scale $C^{0,1}$ and $C^{1,1}$ estimates]
\label{l.C01C11}
There exists $C(d,\Lambda)<\infty$ such that, for every~$R\in [2,\infty)$ and solution $u\in H^1(B_R)$ of the equation
\begin{equation}
\label{e.pde3}
-\nabla \cdot \a\nabla u = 0\quad \mbox{in} \ Q_R, 
\end{equation}
we have, for every $r\in [1,R]$
\begin{equation}
\label{e.C01}
\left\| u - \hat{u}(0) \right\|_{\underline{L}^2(Q_r)}
\leq 
\frac{Cr}{R} \left\| u \right\|_{\underline{L}^2(Q_R)}
\end{equation}
and, with $\psi \in \A_{1}$ such that $\hat{\psi}(0) = \hat{u}(0)$ and $D\hat \psi(0) = D\hat{u}(0)$,  
\begin{equation}
\label{e.C11}
\left\| u - \psi \right\|_{\underline{L}^2(Q_r)}
\leq 
\frac{Cr^2}{R^2} \left\| u \right\|_{\underline{L}^2(Q_R)}.
\end{equation}
\end{lemma}
\begin{proof}
Assume that~$R\geq 4$ and~$u\in H^1(B_R)$ is a solution of~\eqref{e.pde3}. 

\smallskip

\emph{Step 1.} The proof of~\eqref{e.C01}. 
By~\eqref{e.diffbound} for $m=1$, 
\begin{align}
\label{e.applydiffbounds}
\left\| Du \right\|_{L^\infty(Q_{3R/4})}
\leq 
\frac{C}{R} \left\| u \right\|_{\underline{L}^2(Q_R)}. 
\end{align}
This implies that, for every $r\in [1,\tfrac12R]$, 
\begin{equation}
\sup_{z\in Q_r} 
\left| \hat{u}(z) - \hat{u}(0) \right|
\leq 
Cr \sup_{z\in Q_{r+1}} \left| D\hat{u}(z) \right| 
\leq
\frac{Cr}{R} \left\| u \right\|_{\underline{L}^2(Q_R)}.
\end{equation}
Combining this with~\eqref{e.jumpdown}, we obtain, for every~$r\in \left[1,\tfrac12 R\right] \cap \N$, 
\begin{align}
\label{e.iterready.0}
\left\| u - \hat{u}(0) \right\|_{\underline{L}^2(Q_r)}
\leq
\frac{Cr}{R} \left\| u \right\|_{\underline{L}^2(Q_R)}
+ 
C \left\| \nabla u \right\|_{\underline{L}^2(Q_r)}.
\end{align}
Testing~\eqref{e.pde3} with $\phi^2u$, where~$\phi\in C^\infty_c(\Rd)$ is chosen to satisfy $\indc_{Q_r} \leq \phi \leq \indc_{Q_{r+s}}$ and $\| \nabla \phi \|_{L^\infty(\Rd)} \leq Cs^{-1}$, we obtain, for every $r\in \left[1,\tfrac12 R \wedge (R-s)\right]$, 
\begin{equation}
\label{e.cacc.your.u}
\left\| \nabla u \right\|_{\underline{L}^2(Q_r)}^2
\leq
\frac{C}{s} 
\left\| u \right\|_{\underline{L}^2(Q_{r+s})}^2.
\end{equation}
Requiring $s=s(d,\Lambda)<\infty$ to be large enough that $Cs^{-1}\leq \frac12$ and combining the result with~\eqref{e.iterready.0}, we obtain, for every $r\in \left[1,\tfrac12 R \wedge(R-s)\right]$, 
\begin{equation}
\label{e.iterready}
\left\| u - \hat{u}(0) \right\|_{\underline{L}^2(Q_r)}
\leq 
\frac{Cr}{R} \left\| u \right\|_{\underline{L}^2(Q_R)}
+
\frac12 \left\| u - \hat{u}(0) \right\|_{\underline{L}^2(Q_{r+s})}.
\end{equation}
An iteration of~\eqref{e.iterready} implies that~\eqref{e.C01} is valid for every $r\in \left[1,\frac12 R \wedge (R-s)\right]$, which also requires that $R \geq s+1$. Since $s\leq C$, 
these restrictions may be removed by enlarging the constant~$C$ in~\eqref{e.C01}. 

\smallskip

\emph{Step 2.} The proof of~\eqref{e.C11}. 
We suppose that~$R\geq 4+s$.  
Applying~\eqref{e.diffbound} for $m=2$, we have, for every~$r\in\left[1,\tfrac12 R\right]$,
\begin{align}
\label{e.applydiffbounds.11}
\left\| D^2u \right\|_{L^\infty(Q_{3R/4})}
\leq 
\frac{C}{R^2} \left\| u \right\|_{\underline{L}^2(Q_R)}. 
\end{align}
By~$(\hat{u}-\hat{\psi})(0)=0$, $D(\hat{u}-\hat{\psi})(0)=0$ and $D^2\psi = 0$, for every $r\in [1,R-2]$,  
\begin{equation}
\sup_{z\in Q_r} 
\left| \hat{u}(z)-\hat{\psi}(z) \right|
\leq 
Cr^2 \sup_{z\in Q_{r+2}} \left| D^2\hat{u}(z) \right| 
\leq
\frac{Cr^2}{R^2} \left\| u  \right\|_{\underline{L}^2(Q_R)}.
\end{equation}
Applying~\eqref{e.jumpdown} and using~\eqref{e.cacc.your.u} again (with $u-\psi$ in place of $u$), we deduce that, for every $r \in [1,R-2-s]$, 
\begin{align}
\left\| u -\psi \right\|_{\underline{L}^2(Q_r)} 
&
\leq
\frac{Cr^2}{R^2} \left\| u \right\|_{\underline{L}^2(Q_R)}
+
\frac12 \left\| u -\psi \right\|_{\underline{L}^2(Q_{r+s})}.
\end{align}
An iteration of this inequality implies, for every $r \in [1,R-2-s]$, 
\begin{equation}
\label{e.C11a}
\left\| u -\psi \right\|_{\underline{L}^2(Q_r)} 
\leq 
\frac{Cr^2}{R^2} 
\left\| u \right\|_{\underline{L}^2(Q_R)}.
\end{equation}
The conditions~$r \leq R-2-s$ and~$R\geq s+2$ can be removed by enlarging the constant~$C$ on the right side of~\eqref{e.C11a}. This completes the proof. 
\end{proof}

%
%
%
%
%
%
%
%
%
%

The proof of Theorem~\ref{t.analyticity} for $m\geq 2$ follows along the same lines as the one of Lemma~\ref{l.C01C11}, but is more involved. We take a heterogeneous polynomial~$\psi\in \A_m$ satisfying $D^k \hat{\psi} (0) = D^k \hat{u}(0)$ for every $k\in\{0,\ldots,m\}$, and we endeavor to show that~\eqref{e.Cm1} holds. Compared to the case $m\in\{0,1\}$, the additional difficulty in the case $m\geq 2$ is that the polynomial $p:= -\nabla \cdot \a\nabla \psi$ does not vanish, in general, and must be estimated. This turns out to be the most difficult part of the argument and is the content of Lemma~\ref{l.gaussian.cacc}, below.

\smallskip

Dealing with this difficulty forces us to make some modifications to our strategy of a more technical nature. In particular, we need to work with heat kernels rather than balls or cubes and use some properties of Hermite polynomials. We denote
\begin{equation*} 
	 \Phi_{t,y}(x) :=  (4 \pi t)^{-\frac d2} \exp \left( - \frac{|x-y|^2}{4t} \right).
\end{equation*}
We begin the proof of Theorem~\ref{t.analyticity} for~$m \geq 2$ by presenting a variant of Lemma~\ref{l.jumpdown} for heat kernels (rather than cubes). 

\begin{lemma}
\label{l.jumpdown2}
Let $r \in \N$ and $u \in H^1(Q_r)$. There exists a constant $C(d)<\infty$ such that, for every $y \in \R^d$ and $t \in [1,\infty)$, 
\begin{equation}
\label{e.jumpdown2}
\int_{Q_r} u^2 \Phi_{t,y}  
\leq 
2 \int_{Q_r}  \hat{u}^2  \Phi_{t,y}  
+
C  \int_{Q_r} \left| \nabla u \right|^2 \Phi_{t + 1 ,y} 
\end{equation}
\end{lemma}
\begin{proof}
Observe that, for every $t \geq 1$, $x,x',y\in\Rd$ with $|x-x'| \leq \sqrt{d}$, we have 
\begin{align*}
\frac{\Phi_{t,y} (x')}{\Phi_{t+1,y}(x)}
=
\exp\left( -\frac{|x'-y|^2}{4t(t+1)}
+ \frac{(x-x')\cdot(x+x'-2y)}{4(t+1)}
 \right) \leq C. 
\end{align*}
Thus, by the Poincar\'e inequality, 
\begin{align*}
\int_{z + Q_1} \left| u(x) - \hat u(z)  \right|^2 \Phi_{t,y}(x)\,dx
&
\leq 
C \sup_{x\in z+Q_1} \Phi_{t,y}(x) 
\int_{z + Q_1} \left| \nabla u \right|^2
\\ & 
\leq 
C\int_{z + Q_1} \left| \nabla u(x) \right|^2 \Phi_{t+1,y} (x)\,dx.
\end{align*}
Summing over $z\in \Zd\cap Q_r$ yields
\begin{align*}
\int_{Q_r} \left| u(x) - \hat u(z)  \right|^2 \Phi_{t,y}(x)\,dx
\leq 
C\int_{Q_r} \left| \nabla u(x) \right|^2 \Phi_{t+1,y} (x)\,dx
\end{align*}
Using the triangle inequality, we get~\eqref{e.jumpdown2}.
\end{proof}

The next lemma gives two simple estimates for polynomials which are needed in the proof of Lemma~\ref{l.gaussian.cacc}. 

\begin{lemma}[Polynomial estimates] \label{l.polynomial}
For every $y \in \Rd$, $t \in (0,\infty)$, $m\in \N$ and polynomial $p \in \mathbb{P}_m$, we have
\begin{equation}
\label{e.polystupid1}
\int_{\R^d} \left( \frac{|\nabla ( p  \Phi_{t,y})|}{ \Phi_{t,y}}  \right)^2  \Phi_{t,y} \leq \frac{2d(m+1)}{t} \int_{\R^d} p^2  \Phi_{t,y}  
\end{equation}
and
\begin{equation}  \label{e.polystupid2}
	\int_{\R^d} p^2  \Phi_{t,y}  \leq 2 \int_{ Q_{ 4 \sqrt{m t} }(y)}  p^2  \Phi_{t,y} .
\end{equation}
\end{lemma}
\begin{proof}
It suffices by Fubini to prove the result in dimension~$d=1$. By scaling and translation we may assume that $t = 1/4$ and $y= 0$. We denote $\Phi(x) := \Phi_{1/4,0}(x) = \pi^{-\frac12} \exp(-x^2)$ and represent~$p$ using Hermite polynomials as
\begin{equation*} 
p(x) 
= 
\sum_{k = 0}^m c_k h_k(x),
\end{equation*}
where $h_k$ is the $k$th (physicist) Hermite polynomial, defined by
\begin{equation}
h_k(x):= 
(-1)^k \exp\left(x^2\right) 
\left( \frac{d}{dx}\right)^n
\exp\left(-x^2\right).
\end{equation}
We note that these Hermite polynomials satisfy
\begin{equation}
\label{e.hk.one}
( h_n  \Phi)' 
= ( h_n' - 2 x h_n) \Phi 
= - h_{n+1} \Phi
\end{equation}
and
\begin{equation}
\label{e.hk.two}
\left\{ 
\begin{aligned}
& h_{n+1}(x)  = 2x h_n(x) - 2nh_{n-1}(x),
\\ & 
h_0 = 1, \ h_1(x) = 2x. 
\end{aligned}
\right. 
\end{equation}
They are orthogonal with respect to~$\Phi(x)\,dx$ and satisfy, for every $m,n\in\N$, 
\begin{equation}
\label{e.Hermite.orth}
\int_{\R} h_m(x) h_n(x) \Phi(x)\,dx 
=
\left\{
\begin{aligned}
& 2^m m! & \mbox{if} \ n=m, \\
& 0 & \mbox{if} \ n\neq m. 
\end{aligned}
\right. 
\end{equation}

\smallskip

Using~\eqref{e.hk.one} and~\eqref{e.Hermite.orth}, we see that
\begin{align*}
\int_{\R}
\left(
\frac{(p\Phi)'}{\Phi}\right)^2 \Phi  
&
=
\int_{\R}
\left( \sum_{k=0}^m c_k \frac{( h_k \Phi)'}{\Phi} \right)^2 \Phi
\\&
=
\sum_{k=0}^m c_k^2
\int_{\R} h_{k+1}^2 \Phi
\\&
= 
\sum_{k=0}^m c_k^2 2(k+1)
\int_{\R} h_{k}^2 \Phi
\leq
2(m+1) \int_{\R} p^2\Phi.
\end{align*}
This yields~\eqref{e.polystupid1}.

\smallskip

We turn to the proof of~\eqref{e.polystupid2}. We first observe that, for every $k\in\N$, 
\begin{equation}
\label{e.hk.tailbound}
\left| h_k(x) \right| \leq 4^k|x|^k \quad \forall |x| \geq k^{\frac12}.
\end{equation}
We can prove~\eqref{e.hk.tailbound} by induction and~\eqref{e.hk.two}: if it holds for $k\in \{n-1,n\} $, then 
\begin{align*}
\left| h_{n+1}(x) \right|
&
\leq 
2|x| \left| h_n(x) \right| + 2n \left| h_{n-1}(x) \right|
\\ & 
\leq 
2|x|\cdot 4^n |x|^n
+
2n\cdot 4^{n-1}|x|^{n-1}
\\ & 
\leq 
4^{n+1} |x|^{n+1}. 
\end{align*}
It clearly holds for $k\in\{0,1\}$, so it is valid for all $k$. We deduce from~\eqref{e.hk.tailbound} that, for fixed $\alpha\geq 1$ and every $k\in\{0,\ldots,m\}$, 
\begin{align*}
\int_{|x| > \sqrt{\alpha m}} h_k^2 \Phi 
\leq 
2^{4k+1} \int_{\sqrt{\alpha m}}^\infty x^{2k} \exp\left(-x^2\right) \,dx
\leq
2^{6k+2} k! \exp\left( -\tfrac12 \alpha m \right)
\end{align*}
Thus, using also~\eqref{e.Hermite.orth},
\begin{align*}
\int_{|x|>\sqrt{\alpha m}} p^2 \Phi
&
\leq
(m+1) 
\sum_{k=0}^m
c_k^2
\int_{|x|>\sqrt{\alpha m}} 
h_k^2 \Phi 
\\ &
\leq
(m+1) \exp\left( -\tfrac12 \alpha m \right)
\sum_{k=0}^m
c_k^2
2^{6k+2} k! 
\\ & 
=
(m+1) \exp\left( -\tfrac12 \alpha m \right)
\sum_{k=0}^m
c_k^2 2^{5k+2}
\int_{\R}
h_k^2 \Phi
\\ & 
\leq
(m+1)2^{5m+2} \exp\left( -\tfrac12 \alpha m \right)
\int_{\R} p^2\Phi .
\end{align*}
Taking $\alpha=16$ so that $(m+1)2^{5m+2} \exp\left( -\tfrac12 \alpha m \right) \leq \frac14$ for all $m\geq 1$, we then obtain~\eqref{e.polystupid2} by the triangle inequality.
\end{proof}

The next lemma contains the key additional step needed to adapt the argument of Lemma~\ref{l.C01C11} to~$m\geq 2$. The interesting point is that the polynomial~$p$ does not appear on the right side of the estimate~\eqref{e.gaussian.cacc}.

\begin{lemma} 
\label{l.gaussian.cacc}
For each $\delta\in \left(0,1\right]$, there exist
$C(\delta,d,\Lambda)<\infty$ and $c(d,\Lambda)>0$ such that, for every $m\in\N$, $R \in [Cm,\infty)$, $t\in [Cm,R]$, $y\in Q_{R/2}$, $p\in\mathbb{P}_m$ and solution~$u\in H^1(Q_R)$ of the equation
\begin{equation}
\label{e.pde.withp}
-\nabla \cdot \a \nabla u = p 
\quad \mbox{in} \ Q_R,
\end{equation}
we have the estimate 
\begin{align} 
\label{e.gaussian.cacc}
\int_{Q_{3R/4}}\left(
 \left| \nabla u  \right|^2
+
p^2
\right)  \Phi_{t,y}
& 
\leq  
\delta \int_{Q_{3R/4}} u^2 \Phi_{t+C,y} 
+
\exp\left( -c R \right) \int_{Q_{R}} u^2.
\end{align}
\end{lemma}
\begin{proof}
Let $\delta \in \left( 0,1\right]$ and~$\psi \in W_0^{1,\infty} (Q_R)$ be a cutoff function satisfying 
\begin{equation}
\label{e.cutoff.gc}
\indc_{Q_{3R/4}} \leq \psi \leq \indc_{Q_{R}}
\quad \mbox{and} \quad 
|\nabla \psi| \leq \frac{8}{R}.
\end{equation}
We first test the equation~\eqref{e.pde.withp} with
$u \psi^2 \Phi_{t,y}$ to obtain
\begin{align*}
\int_{Q_R} 
p u \psi^2 \Phi_{t,y} 
&
=
\int_{Q_R}
\left(  \psi^2 \Phi_{t,y} \nabla u\cdot \a \nabla u   + u \nabla ( \psi^2 \Phi_{t,y}) \cdot \a \nabla u \right)
\\ & \notag
\geq 
\frac12 \int_{Q_R}
\left| \nabla u  \right|^2 \psi^2 \Phi_{t,y} 
-C \int_{Q_R} u^2 \Phi_{t,y} \left( \psi^2  \frac{\left| \nabla  \Phi_{t,y}  \right|^2 }{\Phi_{t,y}^2 } + | \nabla \psi|^2 \right) .
\end{align*}
Rerranging and using Young's inequality, we get
\begin{align}
\label{e.testwithu}
\lefteqn{
\int_{Q_R}
| \nabla u|^2 \psi^2 \Phi_{t,y}  
} \qquad & 
\\ & \notag
\leq
\!\!\int_{Q_R} u^2 \,\bigg( 
\frac{\delta}{8}\psi^2 + C
\psi^2  \frac{\left| \nabla  \Phi_{t,y}  \right|^2 }{\Phi_{t,y}^2 } + C | \nabla \psi|^2 \bigg)\,\Phi_{t,y} 
+ \frac{2}{\delta} \int_{Q_R} p^2 \psi^2 \Phi_{t,y}.
\end{align}
Testing~\eqref{e.pde.withp} with $p\psi^2 \Phi_{t,y}$ and using Young's inequality, we get
\begin{align}
\label{e.testwithp}
\int_{Q_R} \! 
p^2  \psi^2 \Phi_{t,y} 
&
= 
\int_{Q_R} 
\a \nabla u \cdot \left( \psi^2 \nabla ( p \Phi_{t,y})   + 2\Phi_{t,y}  p \psi \nabla\psi \right)  
\\ & \notag
\leq 
\frac \delta{16} \!
\int_{Q_R} \!
| \nabla u|^2  \psi^2 \Phi_{t,y} 
+
\frac{C}{\delta} \!
\int_{Q_R} \!\!
\Phi_{t,y}\bigg(  \frac{ \left| \nabla  ( p \Phi_{t,y} )  \right|^2 }{\Phi_{t,y}^2 }  \psi^2  + p^2 |\nabla \psi|^2  \bigg).
\end{align}
Combining the two previous displays, we obtain
\begin{align}
\label{e.tricky.combine}
\int_{Q_R}
\left(
 \left| \nabla u  \right|^2
+
\frac1\delta p^2
\right) 
\psi^2 \Phi_{t,y} 
&
\leq 
 \int_{Q_R} u^2 \left( \frac{\delta}{4}+ C  \frac{\left| \nabla  \Phi_{t,y}  \right|^2 }{\Phi_{t,y}^2 }  \right) \psi^2 \Phi_{t,y} 
\\ & \qquad \notag
+
\frac{C}{\delta^{2}}
\int_{Q_R}
\left(  \frac{ \left| \nabla  ( p \Phi_{t,y} )  \right|^2 }{\Phi_{t,y}^2 }  \right) \psi^2  \Phi_{t,y} 
\\ & \qquad \notag
+ 
C \int_{Q_R}  \left(u^2 + \frac1{\delta^{2}} p^2 \right) 
| \nabla \psi|^2 \Phi_{t,y}.
\end{align}
We next estimate the three terms on the right. 

\smallskip

Observe that, for every $s \in (0,t]$, we have 
\begin{equation*}
\frac{ \left| \nabla\Phi_{t,y}(x)  \right|^2 }{\Phi_{t,y}^2(x) }
\frac{\Phi_{t,y}(x)}{\Phi_{t+s,y}(x)}
\leq
\frac{|x-y|^2}{4t^2}  \frac{\Phi_{t,y}(x)}{\Phi_{t+s,y}(x)} 
\leq 
\frac{2}{s} \left( \frac{s |x-y|^2}{8t^2} \exp \left( -  \frac{s|x-y|^2}{8t^2} \right) \right) 
\leq 
\frac{1}{s}.
\end{equation*}
Therefore, we can find~$C(\delta,d,\Lambda)<\infty$ sufficiently large that, for every~$y \in Q_{R/2}$ and $t \in \left[ C, R\right]$,  we can estimate the first term on the right side of~\eqref{e.tricky.combine} by
\begin{align}
\label{e.fatten.up}
\int_{Q_{3R/4}} u^2 \left( \frac{\delta}{4}+ C  \frac{\left| \nabla  \Phi_{t,y}  \right|^2 }{\Phi_{t,y}^2 }  \right) \psi^2 \Phi_{t,y}  & 
\leq 
\frac{\delta}{2} \int_{Q_{3R/4}} u^2 \Phi_{t+C,y}
\end{align}
Since $y \in Q_{R/2}$ and the fact that $\nabla \psi$ is supported in $Q_R \setminus Q_{3R/4}$, we estimate
\begin{multline}
\label{e.glotten.up}
\int_{Q_R \setminus Q_{3R/4}} 
u^2 \left( \frac{\delta}{2}+ C  \frac{\left| \nabla  \Phi_{t,y}  \right|^2 }{\Phi_{t,y}^2 }  \right) \psi^2 \Phi_{t,y}  
+ \int_{Q_R}
 (u^2 + \delta^{-2}p^2) |\nabla \psi|^2  \Phi_{t,y}
\\
\leq 
C \exp\left( - \frac{R^2}{Ct} \right)
\int_{Q_R} (u^2 + p^2) .
\end{multline}
Turning to the second term on the right side of~\eqref{e.tricky.combine}, we use Lemma~\ref{l.polynomial} to obtain, for every
$y \in Q_{R/2}$ and $t \in \left( 0 , (64 m)^{-1} R^2 \right]$ (note that $Q_{4 \sqrt{mt}}(y) \subset Q_{3R/4}$),
\begin{align}
\label{e.polyineq}
\int_{\R^d}
\frac{|\nabla (p\Phi_{t,y}) |^2}{\Phi_{t,y}^2} \Phi_{t,y} 
\leq
\frac{4dm}{t} \int_{\Rd} p^2 \Phi_{t,y}
\leq
\frac{8dm}{t} \int_{Q_{3R/4}} p^2 \Phi_{t,y}.
\end{align}
Combining the above inequalities~\eqref{e.tricky.combine},~\eqref{e.fatten.up},~\eqref{e.glotten.up} and~\eqref{e.polyineq}, we obtain 
\begin{align*}
\lefteqn{
\int_{Q_{3R/4}}
\Phi_{t,y} 
\left(
 \left| \nabla u  \right|^2
+
\frac1\delta p^2
\right) 
} \qquad & 
\\ & \notag
\leq
\frac{\delta}{2} \int_{Q_{3R/4}} u^2 \Phi_{t+C,y}
+
C \exp\left(-\frac{R^2}{Ct} \right) 
\int_{Q_{R}} \left( u^2 +p^2 \right)
+
\frac{Cm}{\delta^2 t} \int_{Q_{3R/4}} \Phi_{t,y} p^2 
.
\end{align*}
If $t \in [ 2C\delta^{-1} m,R]$, then the last term on the right may be reabsorbed on the left side, giving us the estimate
\begin{equation}
\label{e.tricky.approach}
\int_{Q_{3R/4}} \!\!
\Phi_{t,y} 
\left(
 \left| \nabla u  \right|^2
+
p^2
\right) 
\leq
\delta \int_{Q_{3R/4}} u^2 \Phi_{t+C,y}
+
 C \exp\left(-cR \right) 
\int_{Q_{R}} \left( u^2 +p^2 \right).
\end{equation}
It remains to estimate the second term on the right side of~\eqref{e.tricky.approach}. 

\smallskip

We proceed as above, this time testing the equation~\eqref{e.pde.withp} with $u \psi^2$ and $p \psi^2$ (but without the factor of $\Phi_{t,y}$). We obtain, respectively, for every $\theta \in (0,\infty)$, 
\begin{equation*} 
\int_{Q_{R}} |\nabla u|^2 \psi^2 \leq \frac{\theta R^2}{2}  \int_{Q_{R}} p^2 \psi^2  + C  \int_{Q_R} \left(\frac1{\theta R^2} + |\nabla \psi|^2 \right) u^2  
\end{equation*}
and
\begin{equation*} 
\int_{Q_{R}} p^2 \psi^2 
\leq 
\frac1{\theta R^2} \int_{Q_{R}} |\nabla u|^2 \psi^2 
+ 
C \theta R^2 \int_{Q_{R}} |\nabla p|^2 \psi^2  + C\theta \int_{Q_{R}} |\nabla \psi|^2 p^2.
\end{equation*}
Combining these and using the properties of~$\psi$ in~\eqref{e.cutoff.gc}, we obtain
\begin{equation*} 
\int_{Q_{3R/4}} p^2 
\leq 
C \theta R^2 \int_{Q_{R}} |\nabla p|^2 + C\theta  \int_{Q_{R}} p^2  
+ \frac{C}{\theta^2 R^4}\int_{Q_R} u^2 . 
\end{equation*}
Choosing $\theta := C^{-m}$ for $C(d)<\infty$, we may reabsorb the first two terms on the right side of the previous inequality to obtain
\begin{equation}
\label{e.poly.yomama}
\left\| p \right\|_{L^\infty(Q_R)}^2
\leq 
C^m \int_{Q_{3R/4}} p^2 
\leq \frac{C^{m}}{R^4} \int_{Q_R} u^2 . 
\end{equation}
Therefore, if $R \geq Cm$ for $C$ sufficiently large, we obtain
\begin{align}
\label{e.crude.pbound}
 \exp\left(-cR \right) 
\int_{Q_{R}} \left( u^2 +p^2 \right)
&
\leq
\left( 1 + C^{m} \right) \exp(-cR )  \int_{Q_{R}} u^2
\\ & \notag
\leq 
\exp\left(-\tfrac 12c R \right) \int_{Q_R} u^2.
\end{align}
Combining~\eqref{e.tricky.approach} and~\eqref{e.crude.pbound} and shrinking~$c$ yields the lemma. 
\end{proof}

We are now ready to give the proof of Theorem~\ref{t.analyticity}. 

\begin{proof}[{Proof of Theorem~\ref{t.analyticity}}]
We proved the statement of the theorem for $m\in\{0,1\}$ already in Lemma~\ref{l.C01C11}, so we may suppose $m\geq 2$. 
By Lemmas~\ref{l.corr.growth} and~\ref{l.diffbound}, there exists~$\psi \in \mathbb{A}_m$ satisfying 
\begin{equation}
\label{e.yourotherass}
D^k\hat{\psi}(0) = D^k\hat{u}(0) 
\quad
\mbox{for every} \ k\in\{0,\ldots,m\}
\end{equation}
and, for every $r\geq m$,
\begin{align}
\label{e.psi.control}
\left\| \psi \right\|_{\underline{L}^2(Q_r)}
&
\leq
\sum_{k=0}^m
\left( \frac{Cr}{k+1} \right)^k
\left| D^k\hat{u}(0) \right| 
\\ & \notag
\leq 
\sum_{k=0}^m
\left( \frac{Cr}{k+1} \right)^k
\left( \frac{C(k+1)}{r} \right)^k \left\| u \right\|_{\underline{L}^2(Q_r)}
\leq
C^m \left\| u \right\|_{\underline{L}^2(Q_r)}.
\end{align}
For convenience, we denote $v:= u - \psi$. 
Using~\eqref{e.yourotherass} and that $D^{m+1}\psi = 0$, we find that, for every $r\in [2m,\tfrac 12R]$, 
\begin{align}
\label{e.grid.hammer}
\sup_{z\in Q_r} 
\left| \hat{v} (z)  \right|
\leq 
\frac{Cr^{m+1}}{(m+1)!}  \sup_{z\in Q_{r+m+1}} \left| D^{m+1}\hat{u}(z) \right| 
\leq
\left( \frac{Cr}{R} \right)^{m+1} \left\| u \right\|_{\underline{L}^2(Q_R)}.
\end{align} 
By Lemma~\ref{l.jumpdown2}, for every $r\in [Cm,R]$ and $y\in Q_{R/2}$,
\begin{align}
\label{e.jumpdown.applied}
\int_{Q_{3R/4}} v^2 \Phi_{r,y}
&
\leq 
2 \int_{Q_{3R/4}}\hat{v}^2\Phi_{r,y}  
+
C \int_{Q_{3R/4}} \left| \nabla v \right|^2 \Phi_{r + 1,y}.
\end{align}
We may estimate the first term on the right side of~\eqref{e.jumpdown.applied}, using~\eqref{e.grid.hammer}, to get, for each $r\in [Cm,\tfrac1{8}R]$ and $y\in Q_r$, 
\begin{align}
\label{e.layercaking}
\lefteqn{
\int_{Q_{3R/4}} \hat{v}^2 \Phi_{r,y}
\leq 
\int_{Q_{7R/8}} \hat{v}(\cdot-y)^2 \Phi_{r}
} \qquad & 
\\ & \notag
=
\int_0^{2r} \int_{\partial Q_s} \hat{v}(\cdot-y)^2 \Phi_r\,ds
+
\int_{2r}^{7R/8} \int_{\partial Q_s} \hat{v}(\cdot-y)^2 \Phi_r\,ds
\\ & \notag
\leq 
\sup_{Q_{3r}} \hat{v}^2 \int_{Q_r} \Phi_r
+
\int_{2r}^{7R/8} 
\sup_{Q_{s+r}} \hat{v}^2 
\exp\left( -\frac{cs^2}{r}  \right) \,ds
\\ & \notag
\leq
\left\| u \right\|_{\underline{L}^2(Q_R)}^2
\left( 
\left( \frac{Cr}{R} \right)^{2(m+1)}
+
\int_{2r}^{7R/8}
\left( \frac{Cs}{R} \right)^{2(m+1)}
\exp\left( -\frac{cs^2}{r}  \right) \,ds
\right)
\\ & \notag
\leq 
\left( \frac{Cr}{R} \right)^{2(m+1)}
\left\| u \right\|_{\underline{L}^2(Q_R)}^2.
\end{align}
Indeed, the second integral on the fourth line of the display above can be estimated straightforwardly by changing variables as follows:
\begin{align*}
\int_{2r}^{7R/8}
\left( \frac{Cs}{R} \right)^{2(m+1)}
\exp\left( -\frac{cs^2}{r}  \right) \,ds
&
\leq
\left( \frac{Cr}{R} \right)^{2(m+1)} 
\int_1^\infty
s^{2(m+1)} \exp\left( -crs^2 \right)\,ds
\\ & 
\leq
\left( \frac{Cr}{R} \right)^{2(m+1)} 
\int_1^\infty
s^{2(m+1)} \exp\left( -m s^2 \right)\,ds
\\ & 
\leq 
\left( \frac{Cr}{R} \right)^{2(m+1)} 
\int_m^\infty
\frac1{2m} \left( \frac{s}{m} \right)^{m+\frac12} \exp(-s)\,ds
\\ & 
\leq 
\left( \frac{Cr}{R} \right)^{2(m+1)}.
\end{align*}
For second term on the right side~\eqref{e.jumpdown.applied}, we fix $\delta>0$ to be selected below and apply Lemma~\ref{l.gaussian.cacc}, noting that~$-\nabla\cdot \a \nabla v = -\nabla \cdot\a\nabla \psi \in \P_{m-2}$ in $B_R$, to obtain, under the condition that~$Cm\leq r\leq R$ for $C=C(\delta,d,\Lambda)$, the estimate
\begin{equation}
\label{e.applybighammer}
\int_{Q_{3R/4}}
\left(
 \left| \nabla v  \right|^2
+
p^2
\right) 
\Phi_{r + 1 ,y} 
\leq 
\delta \int_{Q_{3R/4}} v^2 \Phi_{r + 1 + C ,y} 
+
\exp\left( - c R \right) 
\left\| v \right\|_{\underline{L}^2(Q_R)}^2.
\end{equation}
Combining~\eqref{e.jumpdown.applied},~\eqref{e.layercaking} and~\eqref{e.applybighammer} and taking $\delta(d,\Lambda)>0$ sufficiently small, we get, for every $r\in [Cm,\tfrac1{8}R]$ and $y\in Q_r$, 
\begin{align*}
\lefteqn{
\int_{Q_{3R/4}} v^2 \Phi_{r,y}
} \qquad & 
\\ &
\leq
\frac12 \int_{Q_{3R/4}} v^2 \Phi_{r + C ,y} 
+
\left( \frac{Cr}{R} \right)^{2(m+1)}
\left\| u \right\|_{\underline{L}^2(Q_R)}^2
+  C \exp\left( - c R \right) \left\| v \right\|_{\underline{L}^2(Q_R)}^2.
\end{align*}
Integrating this over $y\in Q_r$ and using that $|Q_{r+C}| \leq \frac 32|Q_r|$ for $r\geq C$, we obtain, for every $r\in [Cm,\tfrac1{8}R]$,
\begin{align*}
\int_{Q_{3R/4}} v^2 \left( \frac{\indc_{Q_r}}{|Q_r|} \ast \Phi_{r} \right)
&
\leq 
\frac34
\int_{Q_{3R/4}} v^2 \left(\frac{\indc_{Q_{r+C}}}{|Q_{r+C}|} \ast \Phi_{r+C} \right)
\\ & \qquad 
+
\left( \frac{Cr}{R} \right)^{2(m+1)} \left\| u \right\|_{\underline{L}^2(Q_R)}^2
+  C \exp\left( - c R \right) \left\| v \right\|_{\underline{L}^2(Q_R)}^2.
\end{align*}
An iteration now yields, for every $r\in [Cm,\tfrac1{8}R]$, 
\begin{equation}
\int_{Q_{3R/4}} v^2 \left( \frac{\indc_{Q_r}}{|Q_r|} \ast \Phi_{r} \right)
\leq 
\left( \frac{Cr}{R} \right)^{2(m+1)} \left\| u \right\|_{\underline{L}^2(Q_R)}^2
+  C \exp\left( - c R \right)  \left\| v \right\|_{\underline{L}^2(Q_R)}^2.
\end{equation}
We deduce that, for every $r\in [Cm,\tfrac1{8}R]$,
\begin{equation}
\label{e.vboundsesxp}
\left\| v \right\|_{\underline{L}^2(Q_r)} 
\leq
\left( \frac{Cr}{R} \right)^{m+1} \left\| u \right\|_{\underline{L}^2(Q_R)}
+
C \exp\left( - c R \right)  \left\| v \right\|_{\underline{L}^2(Q_R)}.
\end{equation}
By~\eqref{e.psi.control} and $R\geq Cm$, we have that 
\begin{equation*}
\left\| v \right\|_{\underline{L}^2(Q_R)}
\leq
\left\| u \right\|_{\underline{L}^2(Q_R)}
+
\left\| \psi \right\|_{\underline{L}^2(Q_R)}
\leq 
C^m \left\| u \right\|_{\underline{L}^2(Q_R)}.
\end{equation*}
We therefore obtain, for every~$R\geq Cm$ and $r \in [Cm,R]$,
\begin{equation}
\exp\left( - c R \right)  \left\| v \right\|_{\underline{L}^2(Q_R)}
\leq
\exp(-cR) \left\| u \right\|_{\underline{L}^2(Q_R)}
\leq 
\left( \frac{Cr}{R} \right)^{m+1} \left\| u \right\|_{\underline{L}^2(Q_R)}.
\end{equation}
Combining this with~\eqref{e.vboundsesxp} and substituting $v=u-\psi$, we finally obtain, for every $r\in \left[Cm,\tfrac18 R \right]$, 
\begin{equation}
\label{e.upsi.yodama}
\left\| u - \psi \right\|_{\underline{L}^2(Q_r)} 
\leq
\left( \frac{Cr}{R} \right)^{m+1} \left\| u \right\|_{\underline{L}^2(Q_R)}.
\end{equation}
This is~\eqref{e.Cm1}, although the heterogeneous polynomial $\psi$ is not necessary~$\a(x)$--harmonic as in the statement of the theorem. 

\smallskip

To complete the proof, we must therefore replace $\psi$ by an element of~$\A_m$ which is $\a(x)$--harmonic. To do this, we must estimate~$p$. By~\eqref{e.poly.yomama} applied to~$u-\psi$ instead of~$u$, we have, for every $r \in \left[Cm,\tfrac18R\right]$, 
\begin{equation}
\left\| p \right\|_{L^\infty(Q_r)}
\leq 
\frac{C^m}{r^2} \left\| u - \psi \right\|_{\underline{L}^2(Q_r)}
\leq 
\frac1{r^2} \left( \frac{Cr}{R} \right)^{m+1} \left\| u \right\|_{\underline{L}^2(Q_R)}.
\end{equation}
Likewise, for every~$k\in\{1,\ldots,m\}$ and~$r\in\left[Cm,\tfrac1{16}R\right]$, we have 
\begin{align*}
\left\| D^kp \right\|_{L^\infty(Q_r)}
&
\leq 
\frac{C^{m-k}}{r^2} \left\| D^k(u - \psi) \right\|_{\underline{L}^2(Q_r)}
\\ & \notag
\leq 
\frac{C^{m-k}}{r^2} \left( \frac{Cr}{R} \right)^{m+1-k} 
\left\| D^ku \right\|_{\underline{L}^2\left(Q_{R/2} \right)}
\\ & \notag
\leq
\frac{C^{m-k}}{r^2} \left( \frac{Cr}{R} \right)^{m+1-k} 
\left( \frac{Ck}{R} \right)^k
\left\| u \right\|_{\underline{L}^2\left(Q_{R} \right)}
\\ & \notag
\leq
\frac1{r^2} 
\left( \frac{k}{r} \right)^k
\left( \frac{Cr}{R} \right)^{m+1}
\left\| u \right\|_{\underline{L}^2\left(Q_{R} \right)}.
\end{align*}
The first line above is obtained by applying~\eqref{e.poly.yomama} with $D^k(u-\psi)$ and $D^kp$ in place of $u$ and $p$,  
while the second line above is obtained by applying~\eqref{e.upsi.yodama} to $D^ku$ and $D^k\psi$ in place of $u$ and $\psi$ and the third line is an application of Lemma~\ref{l.diffbound}. By Lemma~\ref{l.polyhit.bound}, the bound~\eqref{e.polyhit.bounds.D}, and the previous two displays, we can find $\zeta \in \A_m$ such that $-\nabla \cdot \a\nabla \zeta = p$ and, for every $r \in \left[Cm,\tfrac 1{16}R\right]$,
\begin{equation*}
\left\| \zeta \right\|_{\underline{L}^2(Q_r)}
\leq
Cr^2 \sum_{n=0}^m \left( \frac{Cr}{n+1} \right)^{n} \left| D^n\hat{p}(0) \right|
\leq 
 \left( \frac{Cr}{R} \right)^{m+1} \left\| u \right\|_{\underline{L}^2(Q_R)}.
\end{equation*}
By the triangle inequality, we therefore obtain, for every~$r\in [Cm,\tfrac1{16}R]$, 
\begin{equation*}
\left\| u - \psi -\zeta \right\|_{\underline{L}^2(Q_r)} 
\leq
\left\| u - \psi  \right\|_{\underline{L}^2(Q_r)} 
+
\left\| \zeta \right\|_{\underline{L}^2(Q_r)} 
\leq
\left( \frac{Cr}{R} \right)^{m+1} \left\| u \right\|_{\underline{L}^2(Q_R)}.
\end{equation*}
Enlarging $C$, we obtain this inequality for every~$r\in [Cm,R]$. Renaming $\psi + \zeta$ to be $\psi$, we obtain the theorem. 
\end{proof}

\begin{remark}
\label{r.pick.psi}
Observe that, in addition to proving the statement of  Theorem~\ref{t.analyticity}, we also proved the estimate~\eqref{e.Cm1} for the unique $\psi\in \mathbb{A}_m$ which satisfies
\begin{equation}
D^k\hat{\psi}(0) = D^k\hat{u}(0), \quad \forall k\in\{0,\ldots,m\}. 
\end{equation}
Of course, this $\psi$ is not necessarily $\a(x)$--harmonic, unlike the one in the statement of the theorem. 
\end{remark}

\section{Consequences}
\label{s.conseq}

We conclude the paper by showing that Corollary~\ref{c.analyticity.exp} and Theorems~\ref{t.Liouville},~\ref{t.quc} and~\ref{t.bottom} are easy consequences of Theorem~\ref{t.analyticity} and some of the lemmas which were used to prove it. 

\smallskip

\begin{proof}[{Proof of Corollary~\ref{c.analyticity.exp}}]
Applying Theorem~\ref{t.analyticity} with $m=\left\lfloor \delta R \right\rfloor$, we obtain $\psi\in \A^0_m\subseteq \A_{\lfloor R \rfloor }^0$ such that, for every $s\in [C\delta R, R]$,  
\begin{equation}
\left\| u - \psi \right\|_{\underline{L}^2(Q_s)} 
\leq 
\left( \frac{Cs}{R} \right)^{m+1} 
\left\| u \right\|_{\underline{L}^2(Q_R)} .
\end{equation}
Taking $\delta(d,\Lambda)>0$ sufficiently small and setting $s: = C\delta R$, we obtain
\begin{equation}
\left\| u - \psi \right\|_{\underline{L}^2(Q_s)} 
\leq 
\left( \frac12 \right)^{m+1} \left\| u \right\|_{\underline{L}^2(Q_R)} 
\leq
\exp(-cR) \left\| u \right\|_{\underline{L}^2(Q_R)} . 
\end{equation}
By the $C^{0,1}$ estimate (Lemma~\ref{l.C01C11}), we have that, for every $r\in [1,s]$, 
\begin{equation}
\left\| u - \psi \right\|_{\underline{L}^2(Q_r)} 
\leq
C \left\| u - \psi \right\|_{\underline{L}^2(Q_s)}. 
\end{equation}
This completes the proof. 
\end{proof}

We begin with the proof Theorem~\ref{t.Liouville}.

\begin{proof}[Proof of Theorem~\ref{t.Liouville}]
Assume that $u\in H^1_{\mathrm{loc}}(\Rd)$ satisfies
\begin{equation}
\label{e.pde4}
-\nabla\cdot \a\nabla u = 0 \quad\mbox{in} \ \Rd
\end{equation}
and, for some $\delta\in (0,1]$, 
\begin{equation}
\label{e.slowexp.growth}
\limsup_{r\to \infty}\ \exp\left( - \delta r \right) \left\| u \right\|_{\underline{L}^2(Q_r)} 
\leq 
1.
\end{equation}
We will show that there exist $\delta_0(d,\Lambda)>0$ and $C(d,\Lambda)<\infty$ such that $\delta \leq \delta_0(d,\Lambda)$ implies that $u \in \A_\infty(C\delta)$. Let $C_1$ be the constant $C$ in the statement of Theorem~\ref{t.analyticity}.

\smallskip

By Lemma~\ref{l.diffbound}, for every $m\in\N$ and $r>m+2$, we have
\begin{equation}
\left| D^m\hat{u}(0) \right| 
\leq
\left( \frac{Cm}{r} \right)^m 
\left\| u \right\|_{\underline{L}^2(Q_r)}.
\end{equation}
If $\delta \in (0,\delta_0]$ with $\delta_0(d,\Lambda)>0$ sufficiently small, then we may take $r:=C\delta^{-1} m$ in the previous inequality and use~\eqref{e.slowexp.growth} to obtain, for~$m\in\N$ sufficiently large, 
\begin{equation}
\label{e.Dmhatu.bound}
\left| D^m\hat{u}(0) \right| 
\leq
\delta^m
\exp\left( Cm \right)
\leq 
\left( C\delta \right)^m.
\end{equation}
By Lemma~\ref{l.corr.growth}, for each $m\in\N$, there exists a unique element~$\psi_m$ of~$\A_m$ satisfying
\begin{equation*}
D^k \hat{\psi}_m (0) = D^k \hat{u}(0) \quad \forall k\in\{ 0,\ldots,m\}. 
\end{equation*}
Let $q_m\in\P_m$ be such that $\psi_m=\sum_{k=0}^m \nabla^k q_m:\phi^{(k)}$. 
By~\eqref{e.corr.growth.q} and~\eqref{e.Dmhatu.bound}, we have that, for every $n,m\in\N$ with $n\leq m$, 
\begin{align*}
\left| \nabla^n(q_m-q_{m+1}) (0) \right|
&
\leq 
\left( \frac{n+1}{m} \right)^{n} C^m \left| D^{m+1} \hat{\psi}_m (0) \right| 
\\ & 
\leq 
\left( \frac{n+1}{m} \right)^{n} C^m  (C\delta)^{m+1}
\leq (C\delta)^{m+1}
\end{align*}
and hence, if $\delta_0(d,\Lambda)>0$ is sufficiently small and $\delta \in (0,\delta_0]$,
\begin{equation}
\left| \nabla^n q_m(0) \right|
\leq 
\left| D^n\hat{u}(0) \right|
+
\sum_{k=n}^{m-1}
\left| \nabla^n(q_k-q_{k+1}) (0) \right|
\leq
\sum_{k=n}^{m}
\left( C\delta \right)^k \leq (C\delta)^n. 
\end{equation}
Thus for $\delta \in (0,\delta_0]$, we have that~$\nabla^nq_m(0)$ is uniformly bounded in~$m\geq n$ by $(C\delta)^n$ and converges as $m\to \infty$. In particular, there exists~$q\in \mathbb{P}_\infty(C\delta)$. 
 such that, for every $m,n\in\N$ with $n\leq m$, 
\begin{equation}
\label{e.vermicious.knid}
\left| \nabla^n q_m (0) - \nabla^n q(0)\right| 
\leq 
(C\delta)^m 
\longrightarrow 
0 
\quad \mbox{as} \ m\to \infty. 
\end{equation}
Define~$\psi\in \A_\infty(C\delta)$ by $\psi:=\sum_{k=0}^\infty \nabla^k q:\phi^{(k)}$. By~\eqref{e.psi.easybound} and~\eqref{e.vermicious.knid}, 
\begin{equation}
\label{e.psi.to.psim}
\left\| \psi_m - \psi  \right\|_{\underline{L}^2(Q_m)}
\leq (C\delta)^m.
\end{equation}
We next apply Theorem~\ref{t.analyticity} to get, for $m\in\N$, $R:= 3C_1^2m$ and $r\in \left[C_1m, R \right]$,
\begin{align*}
\left\| u - \psi_m \right\|_{\underline{L}^2(Q_r)}
\leq 
\left( \frac {C_1r}{3C_1^2m} \right)^{m+1} \left\| u  \right\|_{\underline{L}^2(B_R)}.
\end{align*}
Taking $r :=C_1m$ and applying~\eqref{e.slowexp.growth} yields, for all $m\in\N$ sufficiently large, 
\begin{equation}
\left\| u - \psi_m \right\|_{\underline{L}^2(Q_{C_1m} )}
\leq 
\left( \frac13 \right)^{m+1} \exp\left( 2\delta R \right)
=
\exp\left( -m + 4C_1^2m \delta  \right).
\end{equation}
Thus if~$\delta_0(d,\Lambda)>0$ is small enough, we obtain, for all $m$ sufficiently large, 
\begin{equation}
\left\| u - \psi_m \right\|_{\underline{L}^2(Q_{Cm})}
\leq 
\exp\left(-\frac12 m\right).
\end{equation}
Combining this with~\eqref{e.psi.to.psim} and sending $m \to \infty$ yields $u=\psi$.
\end{proof}

We turn next to the proof of Theorem~\ref{t.quc}. 

\begin{proof}[{Proof of Theorem~\ref{t.quc}}]
We may suppose without loss of generality that 
\begin{equation}
\label{e.your.ball.ass}
\left\| u \right\|_{\underline{L}^2(Q_r)} 
\leq
\left\| u \right\|_{\underline{L}^2(Q_R)},
\end{equation}
otherwise the result is immediate from the inequality $\left\| u \right\|_{\underline{L}^2(Q_s)} 
\leq 
C\left\| u \right\|_{\underline{L}^2(Q_R)}$.

\smallskip

Applying Theorem~\ref{t.analyticity} gives, for every $Cm\leq s \leq cR$, 
\begin{equation}
\left\| u - \psi_m \right\|_{\underline{L}^2(Q_s)}
\leq 
\left( \frac{Cs}{R} \right)^{m+1} \left\| u \right\|_{\underline{L}^2(Q_R)}.
\end{equation}
By Lemmas~\ref{l.corr.growth} and~\ref{l.diffbound}, for every $Cm<r<cs$, 
\begin{align}
\label{e.psi.gb}
\left\| \psi_m \right\|_{\underline{L}^2(Q_s)}
&
\leq
\sum_{k=0}^m
\left( \frac{Cs}{k+1} \right)^k
\left| D^k \hat{u}(0) \right|
\\ & \notag 
\leq
\sum_{k=0}^m
\left( \frac{Cs}{r} \right)^k
\left\| u \right\|_{\underline{L}^2(Q_r)} 
\leq
\left( \frac{Cs}{r} \right)^{m}
\left\| u \right\|_{\underline{L}^2(Q_r)}.
\end{align}
Thus, by the triangle inequality, for every $Cm\leq r \leq cs \leq c^2R$, 
\begin{align}
\label{e.trapped}
\left\| u  \right\|_{\underline{L}^2(Q_s)}
&
\leq
\left\| u - \psi_m \right\|_{\underline{L}^2(Q_s)}
+
\left\| \psi_m \right\|_{\underline{L}^2(Q_s)}
\\ & \notag
\leq
\left( \frac{Cs}{r} \right)^{m}
\left\| u \right\|_{\underline{L}^2(Q_r)}
+
\left( \frac{Cs}{R} \right)^{m+1}
\left\| u \right\|_{\underline{L}^2(Q_R)}.
\end{align}
Fix $\theta \in (0,c]$ small enough that, if~$C$ is the largest of the constants in~\eqref{e.trapped}, then~$C\theta \leq \tfrac12$. Applying~\eqref{e.trapped} with $r = \theta s = \theta^2R$ and with $m\in\N$ chosen to be the largest integer such that we have $Cm \leq r$ (ensuring that~\eqref{e.trapped} is valid) and
\begin{equation}
\label{e.cond2}
\left( \frac{C}{\theta} \right)^{2m} 
\leq 
\frac{\left\| u \right\|_{\underline{L}^2(Q_R)}}{4\left\| u \right\|_{\underline{L}^2(Q_r)}}.
\end{equation}
With these choices, we can bound the first term on the right side of~\eqref{e.trapped} by 
\begin{equation*}
\left( \frac{Cs}{r} \right)^{m}
\left\| u \right\|_{\underline{L}^2(Q_r)}
=
\left( \frac{C}{\theta} \right)^m
\left\| u \right\|_{\underline{L}^2(Q_r)}
\leq
\frac12
\left\| u \right\|_{\underline{L}^2(Q_r)}^{\frac12}
\left\| u \right\|_{\underline{L}^2(Q_R)}^{\frac12}.
\end{equation*}
For the second term, we break into two cases: either $C(m+1)>r$ or~\eqref{e.cond2} is false for $m+1$. If the former holds, then
\begin{align*}
\left( \frac{Cs}{R} \right)^{m+1}
\left\| u \right\|_{\underline{L}^2(Q_R)}
&
=
\left( C\theta \right)^{m+1} 
\left\| u \right\|_{\underline{L}^2(Q_R)}
\\ & 
\leq 
\left( \frac12 \right)^{m+1} 
\left\| u \right\|_{\underline{L}^2(Q_R)}
\leq 
\exp(-cr) 
\left\| u \right\|_{\underline{L}^2(Q_R)},
\end{align*}
while if the latter holds, then for each $\alpha \in (0,1)$, 
\begin{align*}
\left( \frac{Cs}{R} \right)^{m+1}
\left\| u \right\|_{\underline{L}^2(Q_R)}
&
=
\left( C\theta \right)^{m+1} 
\left\| u \right\|_{\underline{L}^2(Q_R)}
\\ & \notag
 \leq 
  \left( C\theta \right)^{m+1} \left\| u \right\|_{\underline{L}^2(Q_R)}^{1-\alpha} \left( \frac{C}{\theta} \right)^{2 \alpha m} \left\| u \right\|_{\underline{L}^2(Q_r)}^\alpha
\\ & 
=
C\theta  
\left( C^{1+2\alpha} \theta^{1-2\alpha} \right)^{m} 
\left\| u \right\|_{\underline{L}^2(Q_R)}^{1-\alpha} 
\left\| u \right\|_{\underline{L}^2(Q_r)}^\alpha.
\end{align*}
Using that
\begin{equation*} 
	C^2 \theta^{\frac12-\alpha} \leq 1 \quad \implies \quad C\theta \left( C^{1+2\alpha} \theta^{1-2\alpha} \right)^{m}  \leq \frac12,
\end{equation*}
taking $\theta(\alpha,d,\Lambda)\in (0,1)$ small enough and recalling~\eqref{e.your.ball.ass}, we may combine the above to obtain
\begin{align*}
\left\| u  \right\|_{\underline{L}^2(Q_s)}
\leq
\left\| u \right\|_{\underline{L}^2(Q_r)}^{\alpha}
\left\| u \right\|_{\underline{L}^2(Q_R)}^{1-\alpha}
+
\exp(-cr) \left\| u \right\|_{\underline{L}^2(Q_R)}.
\end{align*}
This completes the proof. 
\end{proof}

We conclude with proof of Theorem~\ref{t.bottom}.

\begin{proof}[{Proof of Theorem~\ref{t.bottom}}]
Suppose that $A\in [4,\infty)$ and $\lambda \in \left( 0,\delta_0^2\right]$, with $\delta_0\leq 1$ the constant in Theorem~\ref{t.Liouville} in $d+1$ dimensions, and suppose that $u\in H^1_{\mathrm{loc}}(\Rd)$ satisfies
\begin{equation}
-\nabla \cdot \a\nabla u = \lambda u \quad \mbox{in} \Rd
\end{equation}
and
\begin{equation}
\limsup_{r\to\infty} \,
\exp( Ar )
\left\| u \right\|_{\underline{L}^2(Q_r)}   = 0. 
\end{equation}
We will show that if $\lambda$ is sufficiently small and $A$ is sufficiently large, each depending only on $(d,\Lambda)$, then $u\equiv 0$. We may assume~$u \leq 1$. 

\smallskip

Add a dummy variable to~$u$ by defining
\begin{equation}
v(x,x_{d+1}) = \exp\left(\lambda^{\frac12} x_{d+1}\right) u(x)
\end{equation}
we observe that $v$ is a solution of the equation 
\begin{equation}
-\nabla \cdot \tilde\a \nabla v = 0 \quad \mbox{in} \ \R^{d+1},
\end{equation}
where $\tilde{\a}$ is the $(d+1)\times (d+1)$ matrix defined by
\begin{equation}
\tilde\a(x) := \begin{pmatrix} \a(x) & 0 \\ 0 & 1 \end{pmatrix}.
\end{equation}
Select $R\geq 10$ satisfying
\begin{equation}
\label{e.bare.ball}
\left\| v \right\|_{\underline{L}^2(B_R(2Re_1))}
\leq 
\exp\left( - AR + 2\lambda^{\frac12}R \right)
\leq 
\exp\left( -\tfrac12 AR\right).
\end{equation}
Let $c_1 \in \left(0,\tfrac12\right]$ and $C_1 \in [1,\infty)$ be constants to be selected below such that $c_1C_1<1$. Set $m:= \left\lfloor c_1 R \right\rfloor$ and note that $C_1 m < R$. If $c_1$ is sufficiently small, depending only on $(C_1,d,\Lambda)$, then we may apply Theorem~\ref{t.analyticity} to obtain $\psi\in \A_m^0$ satisfying, for every $S \geq R$ and $r\in \left[ C_1 m, S\right]$, 
\begin{equation}
\label{e.reg.app}
\left\| v - \psi \right\|_{\underline{L}^2(B_r)} 
\leq 
\left( \frac{Cr}{S} \right)^{m+1} 
\left\| v \right\|_{\underline{L}^2(B_{S})}. 
\end{equation}
Here~$\mathbb{A}_m$ is understood to be defined with respect to the coefficients~$\tilde{\a}$. 
By taking $r:= C_1 m$ and $S:= R$ we obtain, for $c_1(C_1,d,\Lambda)>0$ sufficiently small, 
\begin{equation*}
\left\| v - \psi \right\|_{\underline{L}^2(B_{C_1m}(2Re_1))} 
\leq 
\left( \frac{CC_1c_1R}{R} \right)^{m+1} 
\left\| v \right\|_{\underline{L}^2(B_R(2Re_1))}
\leq 
\frac12
\left\| v \right\|_{\underline{L}^2(B_R(2Re_1))}.
\end{equation*}
This implies by the triangle inequality and~\eqref{e.bare.ball} that 
\begin{equation}
\left\| \psi \right\|_{\underline{L}^2(B_{C_1m}(2Re_1))} 
\leq 
2\left( \frac{R}{C_1 m} \right)^{d/2} \left\| v \right\|_{\underline{L}^2(B_R(2Re_1))}
\leq \exp\left( -\tfrac 14AR \right)
\end{equation}
provided that $A$ is large enough. 
If $C_1(d,\Lambda)$ is taken sufficiently large, 
we deduce, by Lemmas~\ref{l.corr.growth} and~\ref{l.diffbound}, for every $r>C C_1m$, 
\begin{align}
\label{e.psi.nothing}
\left\| \psi \right\|_{\underline{L}^2(B_r(2Re_1))}
&
\leq
\sum_{k=0}^m
\left( \frac{Cr}{k+1} \right)^k
\left| D^k \hat{\psi}(2Re_1) \right|
\\ & \notag 
\leq
\sum_{k=0}^m
\left( \frac{Cr}{k+1} \right)^k
\left( \frac{Ck}{C_1m} \right)^k
\left\| \psi \right\|_{\underline{L}^2(B_{C_1m}(2Re_1))} 
\\ & \notag 
\leq
\left( \frac{Cr}{m} \right)^m
\exp\left( -\tfrac14 AR \right).
\end{align}
We now, once and for all, fix $C_1$ so that~\eqref{e.psi.nothing} holds and then fix~$c_1$ as above. 
Taking $r:=3R$ and $S:=C_2R$ in~\eqref{e.reg.app} for $C_2>10$, combining the resulting estimate with~\eqref{e.psi.nothing} and then taking $C_2$ to be a sufficiently large constant, depending only on $(d,\Lambda)$, we deduce that 
\begin{align*}
\left\| v \right\|_{\underline{L}^2(B_{R})} 
& 
\leq
C \left\| v \right\|_{\underline{L}^2(B_{3R}(2Re_1))}
\\ & 
\leq 
C \left\| v - \psi \right\|_{\underline{L}^2(B_{3R}(2Re_1))}
+ 
C \left\| \psi \right\|_{\underline{L}^2(B_{3R}(2Re_1))}
\\ & 
\leq 
\left( \frac{CR}{C_2R} \right)^{m+1}
\left\| v \right\|_{\underline{L}^2(B_{C_2R}(2Re_1))}
+
\left( \frac{CR}{m} \right)^m
\exp\left( -\tfrac14AR \right)
\\ & 
\leq 
\exp\left( -cR + C\lambda^{\frac12} R \right)
+
\exp\left( C R - \tfrac14AR \right).
\end{align*}
Taking $\lambda>0$ sufficiently small and $A>1$ sufficiently large, each depending only on~$(d,\Lambda)$, we obtain
\begin{equation}
\left\| v \right\|_{\underline{L}^2(B_R)}
\leq 
2 \exp\left( -cR \right).
\end{equation}
Sending $R\to \infty$ yields $v \equiv 0$ and thus $u\equiv 0$, completing the argument.
\end{proof}

\subsection*{Acknowledgments}
We thank Jonathan Goodman for showing us a Bloch wave perturbation proof of Theorem~\ref{t.bottom} and Yulia Meshkova for pointing us to the paper of Filonov~\cite{Fil}. C.S.~would like to thank Carlos Kenig and Dana Mendelson for many inspiring discussions and for alerting him to the interest of quantitative unique continuation for periodic operators. S.A.~was partially supported by the NSF award DMS-1700329. T.K.~was supported by the Academy of Finland and the European Research Council (ERC) under the European Union's Horizon 2020 research and innovation programme (grant agreement No 818437). C.S.~was partially supported by NSF award DMS-1712841.

\bibliographystyle{abbrv}
\bibliography{analyticity}

\end{document}